\documentclass[12pt,oneside,english]{amsart}
\usepackage[T1]{fontenc}
\usepackage{fullpage}
\usepackage{mathtools}
\usepackage[ngerman,english]{babel}
\usepackage[latin9]{inputenc}
\usepackage{graphicx}
\usepackage{mathcomp}
\usepackage{amsfonts}
\usepackage{amssymb}
\usepackage{amsmath}
\usepackage{cancel}
\usepackage{pdflscape}
\usepackage{booktabs}
\usepackage{relsize}
\usepackage{footnote}
\usepackage{longtable}
\usepackage{xcolor}
\usepackage{pict2e}
\usepackage{nicefrac}
\usepackage[arrow, matrix, curve]{xy}
\usepackage{mathtools}
\usepackage{tikz}
\usetikzlibrary{shapes.geometric, arrows}
\usetikzlibrary{chains}
\tikzset{node distance=2em, ch/.style={circle,draw,on chain,inner sep=2pt},chj/.style={ch,join},every path/.style={shorten >=4pt,shorten <=4pt},line width=1pt,baseline=-1ex}
\usepackage{babel}
\usetikzlibrary{backgrounds}
\usepackage{blkarray}
\usepackage{multirow}
\usepackage{lunadiagrams-KM-3}
\usepackage[colorlinks,
pdfpagelabels,
pdfstartview = FitH,
bookmarksopen = true,
bookmarksnumbered = true,
linkcolor = blue,
urlcolor = blue,
plainpages = false,
hypertexnames = false,
citecolor = blue] {hyperref}
\usepackage{enumitem}
\usepackage[nameinlink]{cleveref}
\tikzstyle{startstop} = [rectangle, rounded corners, minimum width=3cm, minimum height=1cm,text centered, draw=black, fill=red!30]
\tikzstyle{io} = [rectangle, trapezium left angle=70, trapezium right angle=110, minimum width=3cm, minimum height=1cm, text centered, draw=black, fill=blue!30]
\tikzstyle{process} = [rectangle, minimum width=3cm, minimum height=1cm, text centered, draw=black, fill=orange!30]
\tikzstyle{decision} = [rectangle, minimum width=1cm, minimum height=1cm, text centered, draw=black, fill=green!30]
\tikzstyle{arrow} = [thick,->,>=stealth]
\makeatletter

 \theoremstyle{plain}    
 \newtheorem{thm}{Theorem} 
  \theoremstyle{plain}    
  
 \theoremstyle{plain}    
 \newtheorem{lem}[thm]{Lemma} 
 \theoremstyle{plain}    
 \newtheorem{cor}[thm]{Corollary} 
 
 \theoremstyle{remark}
 \newtheorem{rem}[thm]{Remark}
  \theoremstyle{remark}
 
   \theoremstyle{remark}
 \newtheorem{prop}[thm]{Proposition}
 \theoremstyle{definition}
 \newtheorem{defn}[thm]{Definition}
 \theoremstyle{definition}
  \newtheorem{example}[thm]{Example}
    \theoremstyle{definition}
  \newtheorem{lis}[thm]{List}
  \numberwithin{thm}{part}
\makeatother

\newcommand{\hh}{{\mathbf h}}

\definecolor{dgreen}{rgb}{0,0.4,0}

\newcommand{\aaa}{\mathfrak{a}}
\renewcommand{\tt}{\mathfrak{t}}

\newcommand{\N}{\mathbb{N}}
\newcommand{\PP}{\mathbb{P}}

\newcommand{\A}{\mathcal{A}}

\newcommand{\R}{\mathbb{R}}
\newcommand{\Z}{\mathbb{Z}}

\newcommand{\Q}{\mathbb{Q}}
\newcommand{\C}{\mathbb{C}}

\newcommand{\bpm}{\begin{pmatrix}}
\newcommand{\epm}{\end{pmatrix}}
\newcommand{\bi}{\begin{itemize}}
\newcommand{\ei}{\end{itemize}}

\newcommand{\een}{\end{enumerate}}
\renewcommand{\P}{\mathcal{P}}

\renewcommand{\a}{\alpha}
\renewcommand{\aa}{\overline{\alpha}}

\newcommand{\w}{\omega}

\newcommand{\la}{\langle}
\newcommand{\ra}{\rangle}

\newcommand{\G}{\Gamma}

\newcommand{\Hom}{\operatorname{Hom}}
\renewcommand{\L}{\Lambda}
\renewcommand{\S}{\Sigma}

\newcommand{\V}{\mathcal{V}}
\renewcommand{\d}{\delta}
\newcommand{\SL}{\operatorname{SL}}
\newcommand{\GL}{\operatorname{GL}}
\newcommand{\SP}{\operatorname{Sp}}

\newcommand{\SU}{\operatorname{SU}}

\newcommand{\btr}{\mathlarger{\mathlarger{\blacktriangleright}}}
\newcommand{\btl}{\mathlarger{\mathlarger{\blacktriangleleft}}}
\newcommand{\bt}{\mathlarger{\mathlarger{\blacktriangle}}}
\newcommand{\btd}{\mathlarger{\mathlarger{\blacktriangledown}}}

\def\sm#1{S\setminus\{\a_{#1}\}}
\def\ssm#1#2{S\setminus\{\a_{#1},\a_{#2}\}}

\def\Sp#1{\langle -\w, \aa_{#1}^\vee\ra}

\newcommand{\Ssm}{S\setminus}

\newcommand{\sll}{\mathfrak{sl}}
\newcommand{\gll}{\mathfrak{gl}}
\newcommand{\soo}{\mathfrak{so}}
\renewcommand{\sp}{\mathfrak{sp}}
\renewcommand{\gg}{\mathfrak{g}}
\renewcommand{\hh}{\mathfrak{h}}

\begin{document}
\selectlanguage{english}
\title[Moment Polytopes of rank one]{Momentum polytopes of rank one for multiplicity free quasi-Hamiltonian manifolds}
\author{Kay Paulus}
\address{Department of Mathematics, Friedrich-Alexander-University Erlangen-Nuremberg, Cauerstr. 11, 91058 Erlangen}
\email{paulus@math.fau.de}
\date{\today}
\begin{abstract}
 We classify all momentum polytopes of rank one for multiplicity free quasi-Hamiltonian $K$-manifolds for simple and simply connected Lie groups $K$ by using the methods developed in \cite{Kno14}. This leads to lots of new concrete examples of multiplicity free quasi-Hamiltonian manifolds or equivalently, Hamiltonian loop group actions.\\
\end{abstract}

\maketitle

\selectlanguage{english}

\part{Introduction}
Let $K$ be a simply connected compact Lie group with Lie algebra $\mathfrak{k}$ and complexification $G$,   $T_\R$ a maximal torus for $K$ and $T$ its complexification, let $\mathfrak{t}$ be the Cartan of $\mathfrak{k}$.

A quasi-Hamiltonian manifold $M$ is a concept introduced by Alekseev, Malkin and Meinrenken in \cite{AMM98} as a setting equivalent to Hamiltonian Loop group actions.\\

We choose a scalar product on $\mathfrak{k}$, let $\tau$ be an automorphism of $K$ leaving this scalar product on $\mathfrak{k}$ invariant. There are two canonical 1-forms $\theta$ and $\bar{\theta}$ on $\mathfrak{k}$ defined by $\theta(k\xi)=\xi=\bar \theta(\xi k)$ with $k\in K$ and $\xi \in \mathfrak{k}$. Then we take $\Theta_\tau=\frac 12 (\bar \theta + ^{\tau^{-1}}\theta)$ with $\Theta_\tau(k\xi)=\frac 12 (\operatorname{Ad}(k) \xi+^{\tau^{-1}}\xi)$ where $^\tau k$ denotes the operation of $\tau$ on $k\in K$. We use the scalar product on $\mathfrak{k}$ to define the canonical bi-invariant closed 3-form
\[
 \chi:=\frac{1}{12}\la \theta,[\theta,\theta]\ra_{\mathfrak{k}}=\frac{1}{12}\la \bar\theta,[\bar\theta,\bar\theta]\ra_{\mathfrak{k}}
\]
We write $K\tau$ for the twisted $K$-action, i.e. $x\mapsto kx\tau(k)^{-1}$ to avoid confusions with the adjoint action of $K$.

\begin{defn}[\cite{Kno14}, Definition 2.3.]
A quasi-Hamiltonian $K\tau$-manifold is a smooth manifold $M$ equipped with a $K$-action, a 2-form $w$, and a smooth map $m:M\to K\tau$, called the (group valued) moment map which have the following properties:
 \begin{enumerate}[label=\alph*)]
  \item $m$ is $K$-equivariant
  \item the form $w$ is $K$-invariant and satisfies $dw=-m^* \chi$
  \item $w(\xi x,\eta)=\la \xi, m^*\Theta_\tau (\eta)\ra_{\mathfrak{k}}$ for all $\xi \in \mathfrak{k}, \eta\in T_x M$
  \item $\operatorname{ker}w_x=\{\xi x \in T_x M \mid \xi \in  \mathfrak{k}$ with $^{m(x)\tau}\xi+\xi=0\}$
 \end{enumerate}
\end{defn}

We call a quasi-Hamiltonian manifold (from now on: q-Hamiltonian) {\bf multiplicity free} if its symplectic reductions are discrete, i.e. points. Being multiplicity free implies that the image of the moment map $\P_M$ is convex and locally polytopal.\\

For this paper, a  ``quasi-Hamiltonian $K\tau$-manifold'' is always  multiplicity free and compact. This is the setting of \cite{Kno14}. Knop showed that quasi-Hamiltonian manifolds are uniquely determined by their moment polytope $\P$ and principal isotropy group $L_S$ which can be encoded in some lattice $\L_S$. Knop also answers which pairs $(\P,\L_S)$ can arise:

Let $Z$ be a smooth affine spherical variety. We denote by $\aaa_\P\subseteq \mathfrak{t}$ the affine space spanned by $\P$. Any $a\in \P$ gives rise to an isotropy group $K(a)$ with respect to the twisted action. We denote by $C_Z$ the convex cone generated by the weight monoid of $Z$ and by $C_a\P$ the tangent cone of $\P$ in $a\in \A$.

\begin{defn}\label{sphpair}(cf. \cite{Kno14}, Definition 6.6)
Let $K$ be a simply connected compact Lie group with automorphism $\tau$ and fundamental alcove $\A$. Let $\P\subseteq \A$ be a locally convex subset and $\Lambda_S \subseteq \bar\aaa_\P$ a lattice. Then $(\P,\Lambda_S)$ is called {\bf spherical in $a\in \P$}  if
\begin{enumerate}
\item $\P$ is polyhedral in $a$, i.e.
\[
P\cap U = (a+C_a\P)\cap U
\]
for a neighborhood $U$ of $a$ in $\A$, and
\item there is a smooth affine spherical $K(a)_\C$-variety $Z$ with weight monoid $\G_Z$ such that 
\[
C_a\P\cap \Lambda_S = \G_Z
\]
\end{enumerate}
The pair $(\P,\Lambda_S)$ is called a {\bf spherical pair} if it is spherical for all $a\in\P$.\\
The variety $Z$ is called a {\bf local model} in $a$ if it fullfills the conditions above in $a$. By abuse of notation, we shall also call the corresponding weight monoid (cf. \cref{sphericaldata}) and the set of its generators a local model, as a smooth affine spherical variety is uniquely determined by its weight monoid (cf. \cite{Los06}).
\end{defn}

\begin{thm}(\cite{Kno14}, thm. 6.7)
 Let $K$ be a a simply connected compact Lie group with twist $\tau$. Then the map $M\to (\P_M, \Lambda_M)$ furnishes a bijection between
 \begin{itemize}
  \item isomorphism classes of convex multiplicity free quasi-Hamiltonian $K\tau$-manifolds and
  \item spherical pairs $(\P_M, \Lambda_M)$.
 \end{itemize}
Under this correspondence, $M$ is compact if and only if $\P_M$ is closed in $\A$.
\end{thm}

We also recall that every Hamiltonian manifold carries a quasi-Hamiltonian structure.
\begin{defn}
We call a q-Hamiltonian {\bf genuine} if the moment polytope touches every wall of the fundamental alcove at least once. That means that genuine q-Hamiltonian manifolds are exactly those that do not carry a Hamiltonian structure.
\end{defn}

The aim of this paper is to classify all quasi-Hamiltonian manifolds of rank one for $K$ simple, genuine or not.

This paper discusses a part of my nearly-finished doctoral project conducted under the supervision of Friedrich Knop at FAU Erlangen-N\"urnberg. Therefore it clearly has substantial, partially word-by-word, overlap with my doctoral thesis \cite{Pau}.

{\bf Acknowledgment:} \\
I would like to thank Friedrich Knop, Guido Pezzini, Bart Van Steirteghem and Wolfgang Ruppert for countless explanations, discussions and remarks.

\part{Generalities}
\section{Affine root systems}
First we shall recall some well known facts about affine root systems that are needed for the rest of the paper. We follow the introduction in \cite{Kno14} that is mainly based on \cite{Mac72} and \cite{Mac03}.\\

Let $\overline{\aaa}$ be a Euclidean vector space  with an associated affine space $\aaa$. We denote by $A(\aaa)$ the set of affine linear functions on $\aaa$, the gradient of $\a\in A(\aaa)$ is denoted by $\aa \in \overline{\aaa}$, and it has the property
\begin{equation}\label{afflin}
 \a(x+t)=\a(x)+\la \aa, t\ra, ~ x\in \aaa, t\in \overline{\aaa}
\end{equation}

Let $O(\bar\aaa)$ be the orthogonal group of the vector space $\overline\aaa$ and $M(\aaa)$ the isometries of $\aaa$, and let $M(\aaa)\to O(\overline{\aaa}), w\mapsto \overline{w}$ be the natural projection between these two sets. 

A motion $s\in M(\aaa)$ is called a reflection if its fixed point set is an affine hyperplane $H$. We can write this reflection $s_\a(x)=x-\a(x)\aa^\vee$ with the usual convention $\aa^\vee=\frac{2\aa}{||\aa||^2}$. Its action on an affine linear function $\beta\in A(\aaa)$ is:
\[
 s_\a(\beta)=\beta-\la\overline{\beta}, \aa^\vee\ra \a
\]

\begin{defn} 
 An affine root system on an affine space $\aaa$ is a subset $\Phi\subset A(\aaa)$ such that:
 \begin{enumerate}
  \item $\R1\cap \Phi=\emptyset$
  \item $s_\a(\Phi)=\Phi$ for all $\a\in\Phi$
  \item $\la \bar \beta, \bar \alpha ^\vee \ra \in \Z$ for all $\a,\beta \in \Phi$
  \item the Weyl Group $W_\Phi:=\la s_\a,\a\in\Phi\ra\subset M(\aaa)$  is an Euclidean reflection group, which means in particular that $W$ acts properly discontinuously.
  \item The affine root system $\Phi$ is called reduced if $\R\a\cap\Phi=\{+\a, -\a\}$ for all $\a\in\Phi$
 \end{enumerate}
 \end{defn}

Note that  we do not ask $\Phi$ to generate the affine space.  Also, all finite (classical) root systems are affine in this definition.

\begin{defn}
The set of non-decomposable roots of $\Phi$ (with respect to one fixed alcove) is called the set of {\bf simple roots}. We denote the set of all simple roots by $S$.
A root is called positive if it is a linear combination of simple roots with non-negative coefficients.
\end{defn}

We shall name the affine root systems and order the simple roots just as in Tables Aff 1 to Aff 3 in  \cite{Kac90} resp. \cite{Bou81}.

\begin{defn}[\cite{Kno14}, 3.3] 
 Let $\Phi\subset A(\aaa)$ be an affine root system.
\begin{enumerate}
 \item A {\bf weight lattice} for $\Phi$ is a lattice $\Lambda \subseteq A(\aaa)$ with $\bar\Phi\subset \Lambda$ and $\bar\Phi^\vee \subseteq \Lambda^\vee$ where $\Lambda^\vee = \{t \in \bar\aaa \mid \la t, \Lambda \ra \subseteq  \Z\}$ is the dual lattice for $\Phi$.
 \item An {\bf integral root system} is a pair $(\Phi, \Lambda)$ where $\Phi\subset A(\aaa)$ is an affine root system and $\Lambda \subseteq \bar\aaa$ is a weight lattice for $\Phi$.
\end{enumerate}
\end{defn}

The connected components of the complement of the union of all reflection hyperplanes in $\aaa$ are called the {\bf chambers} of $\Phi$ (or $W$). The closure $\bar{\mathcal{A}}$ of a single chamber is called an {\bf alcove}.\\

There is a well known theory of Euclidean reflection groups stating that $W$ acts simply transitively on the set of alcoves, that each alcove is a fundamental domain for the action of $W$ and that the group $W$ is generated by the finitely many reflections about the faces of codimension one of any alcove.

\section{Basic combinatorial invariants of spherical varieties}\label{sphericaldata}
In this section, we recall some basic facts, notations and invariants from the (combinatorial) theory of spherical varieties. We particularly focus on smooth affine spherical varieties of rank one and try to avoid digging into the depth of the general theory as much as possible to keep the paper accessible to a broad range of readers. For a more detailed introduction, we refer to \cite{Lun01}. The author also recommends \cite{Pez10} as a general introduction to spherical varieties and \cite{BL10} as an introduction to the combinatorial theory of spherical varieties containing numerous instructive examples. 

Let $B$ be a Borel subgroup of $G$. Remember that an irreducible  $G$-variety is called {\bf spherical} if it is normal and has an open $B$-orbit. A closed subgroup $H$ is called spherical if $G/H$ is spherical.

The {\bf weight monoid} $\G(Z)$ of a complex affine algebraic $G$-variety $Z$ is the set of isomorphism classes of irreducible representations of $G$ in its coordinate ring.

\begin{defn}\label{combin}
Let $Z$ be a spherical $G$-variety with open orbit $G/H$. The basic invariants the theory of spherical varieties uses about $Z$ are:
\begin{enumerate}
 \item The {\bf (weight) lattice} of $Z$, called $\Lambda(Z)$, is the subgroup of the character group of the Borel $B$ (that can be identified with the character group of the maximal torus $T$) consisting of the $B$-weights of $B$-eigenvectors in the field of rational functions $\C(Z)$.
 
\item The rank of $\Lambda(Z)$ is called the {\bf rank} of $Z$.

\item Let $\G$ be a set of dominant weights of $G$. The set of {\bf simple roots orthogonal to $\G$} is
\[
 S^p(\G):=\{\a\in S: \la \lambda, \alpha^\vee \ra =0 \forall \lambda \in \G\}
\]

\item By \cite{Bri90}, the set of $G$-invariant $\Q$-valued discrete valuations of $\C(Z)$, called $\V(Z)$, is a co-simplicial cone. The set of {\bf spherical roots} $\S(Z)$ of $Z$ is the minimal set of primitive elements of $\Lambda(Z)$ such that 
\[
 \V(Z)=\{\eta\in\Hom_\Z(\Lambda(Z),\Q)\mid\la \eta, \sigma \ra \le 0 \forall \sigma \in \S(Z)\}
\]
\end{enumerate}
\end{defn}
Note that these invariants only depend on the open $G$-orbit $G/H$ of $Z$.\\


Let us recall the list of spherical roots of smooth affine spherical varieties that is a subset of Akhiezer's classification of smooth spherical varieties of rank one in \cite{Akh83}. They are known to be exactly the weights of homogeneous smooth affine spherical varieties of rank one.\\
We also give the so-called Luna Diagrams of the spherical roots, see e.g. \cite{BL10} for a closer explanation of this diagram notation. Note that the factors in [] apply to the whole sums and are optional, meaning that the spherical roots with and without this factor exist.

\begin{lis}\label{lunad}{Diagrams of spherical roots}
 \begin{center}
  \begin{longtable}{p{6cm}p{6cm}}
  Spherical root & Diagram\\\endhead
  \hline
  $\a_1$ & $\begin{picture}(-1800,2400)(0,0) \put(0,0){\usebox{\aone}}\end{picture}$\\
  $2\a_1$& $\begin{picture}(-1800,2400)(0,0) \put(0,0){\usebox{\aprime}}\end{picture}$\\
  $[\frac 12]\a+\a'$ & $\begin{picture}(-1800,2400)(0,-900) \put(0,0){\usebox{\vertex}}
						      \put (2700,0){\usebox \vertex}
						      \multiput(0,0)(2700,0){2}{\usebox\wcircle}
						      \multiput(0,-300)(2700,0){2}{\line(0,-1){600}}
						      \put(0,-900){\line(1,0){2700}}
						      \put(900,-1800){\tiny $[\nicefrac12]$}
						      \end{picture}$\\
 $\a_1+\dots+\a_r$ & $\begin{picture}(-1800,2400)(0,0) \put(0,0){\usebox{\mediumam}}\end{picture}$\\
 $[\frac12]\a_1+2\a_2+\a_3$ & $\begin{picture}(-1800,2400)(0,0) \put(0,0){\usebox{\dthree}}
 \put (1200,900){\tiny$[\nicefrac12]$}\end{picture}$\\
 $\a_1+\dots+\a_r$ & $\begin{picture}(-1800,2400)(0,0) \put(0,0){\usebox{\shortbm}}\end{picture}$\\
 $2\a_1+\dots+2\a_r$& $\begin{picture}(-1800,2400)(0,0) \put(0,0){\usebox{\shortbprimem}}\end{picture}$\\
 $[\frac 12]\a_1+2\a_2+3\a_3$ & $\begin{picture}(-1800,2400)(0,0) \put(0,0){\usebox{\bthirdthree}}\put (2700,900){\tiny$[\nicefrac12]$}\end{picture}$\\
 $\a_1+2\a_2+\dots+2\a_{n-1}+\a_r$ & $\begin{picture}(-1800,2400)(0,0) \put(0,0){\usebox{\shortcm}}\end{picture}$\\
 $[\frac12]2\a_1+\dots+2\a_{r-2}+\a_{r-1}+\a_r$ & $\begin{picture}(-1800,2400)(0,0) \put(0,0){\usebox{\shortdm}}\put (-600,900){\tiny$[\nicefrac12]$}\end{picture}$\\
 $\a_1+2\a_2+3\a_3+2\a_4$ & $\begin{picture}(-1800,2400)(0,0)\put(0,0){\usebox{\ffour}}\end{picture}$\\
 $[2]2\a_1+\a_2$ & $\begin{picture}(-1800,2400)(0,0) \put(0,0){\usebox{\gtwo}}\put(-300,600){\tiny [2]}\end{picture}$\\

  \end{longtable}
 \end{center}
\end{lis}
 
\begin{defn}
 We will use the abbreviation $\aa_{i,j}$ for the sum $\aa_i+\aa_{i+1}+\dots+\aa_j$.
\end{defn}

\section{Local models and twisted conjugacy classes}

Every twisted q-Hamiltonian manifold locally looks like a Hamiltonian one (cf. \cite{Kno14}, part 5) and Sjamaar \cite{Sja98} showed that smooth affine spherical varieties are local models for (multiplicity-free) Hamiltonian manifolds. Their general structure is:

\begin{thm}[cf. \cite{KVS05}, cor. 2.2]
 Let $Z$ be a smooth affine spherical $G$-variety. Then $Z\cong G\times ^H V$ where $H$ is a reductive subgroup of $G$ such that $G/H$ is spherical and $V$ is a spherical $H$-module.
\end{thm}

A smooth affine spherical variety corresponds to a triple $(G,H,V)$. Hence a complete classification of smooth affine spherical varieties would require giving all these triples.  We refer to \cite{KVS05}, Examples 2.3 - 2.6 as an illustration of which problems appear. The actual objects of this classification are:

\begin{defn}[\cite{KVS05}, Def. 2.7.]\leavevmode\\
\begin{enumerate}
\item Let $\hh\subseteq \gg$ be semisimple Lie algebras and let $V$ be a representation of $\hh$. For $\mathfrak{s}$, a Cartan subalgebra of the centralizer $\mathfrak{c_g(h)}$ of $\hh$, put $\bar\hh:=\hh \oplus \mathfrak{s}$, a maximal central extension of $\hh$ in $\gg$. Let $\mathfrak{z}$ be a Cartan subalgebra of $\gll(V)^\hh$ (the centralizer of $\hh$ in $\gll(V)$). We call $(\gg,hh,V)$ a {\bf spherical triple} if there exists a Borel subalgebra $\mathfrak{b}$ of $\gg$ and a vector $v\in V$ such that:
\begin{enumerate}
\item $\mathfrak{b}+\bar \hh =\gg$ and 
\item $[(\mathfrak{b}\cap \bar \hh + \mathfrak{z}]v=V$ where $\mathfrak{s}$ acts as via any homomorphism $\mathfrak{s}\to \mathfrak{z}$  on $V$.
\end{enumerate}
\item Two triples $(\gg_i, \hh_i, \V_i), i=1,2$ are {\em isomorphic} if there exist linear bijections $\a:\gg_1\to \gg_2$ and $\beta:V_1\to V_2$ such that:
\begin{enumerate}
\item $\a$ is a Lie algebra homomorphism
\item $\a(\hh_1)=\hh_2$
\item $\b(\xi v)=\alpha(\xi)\beta(v)$ for all $\xi \in \hh_1$ and $b\in V_1$.
\end{enumerate}
\item Triples of the form $(\gg_1\oplus \gg_2, \hh_1\oplus \hh_2, V_1\oplus V_2)$ with $(\gg_i, \hh_i, V_i)\ne(0,0,0)$ are called {\em decomposable}.
\item Triples of the form $(\mathfrak{k},\mathfrak{k}, 0)$ and $(0,0,V)$ are said to be {\em trivial}.
\item A pair $(\gg,\hh)$ of semisimple Lie algebras is called {\em spherical} if $(\gg,\hh, 0)$ is a spherical triple.
\item A spherical triple (or pair) is {\em primitive} if it is non-trivial an indecomposable.
\end{enumerate}
\end{defn}

Let us particularly recall that a summand of the form $(\gg, \hh, 0)$ corresponds to a homogeneous variety and a factor $(\hh,\hh,V)$ corresponds to a spherical module. 

\begin{thm}[\cite{KVS05}, thm. 2.9.]
 If $G\times^H V$ is a smooth affine spherical variety then $(\gg', \hh', V)$ is a spherical triple. Moreover, it follows from the classification in \cite{KVS05} that every spherical triple arises this way.
\end{thm}

We shall add the spherical triples of the local models in \cref{ListHom} and \cref{ham1}.

The local models are closely related to the local root systems from \cite{Kno14}. We shall give a brief overview how they are constructed.

Let us fix some notation: We have a simply connected compact Lie group $K$ that operates on a smooth manifold $M$ by twisted conjugation. Let $\mathfrak{t}\subset \mathfrak{k}$ be a Cartan subalgebra and $\Phi_K$ the corresponding root system. Let us remember some {\bf facts about twisted conjugacy classes} from \cite{Kno14}:

\begin{thm}\label{levi}[cf. \cite{Kno14},thm. 4.2.] 
 Let $K$ be a simply connected compact Lie group and $\tau$ an automorphism of $K$. Then there is a $\tau$-stable maximal torus $T \subseteq K$ and an integral affine root system $(\Phi_\tau, \Lambda_\tau)$ on $\aaa=\mathfrak{t}^\tau$, the $\tau$-fixed part of $\operatorname{Lie} T$, such that the following holds:
 \begin{enumerate}[label=\alph*)]
  \item Let $\operatorname{pr}^\tau:\mathfrak{t}\to \aaa$ be the orthogonal projection. Then $\bar\Phi_\tau=\operatorname{pr}^\tau\Phi(\mathfrak{k,t})$ and $\Lambda_\tau=\operatorname{pr}^\tau\Xi(T)$. Moreover, $\Lambda_\tau$ is also the weight lattice (dual of the coroot lattice) of $\bar\Phi_\tau$.
  \item For any alcove $\A\subseteq \aaa$ of $\Phi_\tau$ the composition
  \[
   c:\A\hookrightarrow\aaa \overset{exp}{\to} K \to K\tau/K
  \]
is a homeomorphism.
 \item For $a\in\A$ let $u:=\exp a\in K$. Then the twisted centralizer
 \[
  K(a):=K_u=\{ k \in K: ku^\tau k^{-1}=u\}=K^{\bar\tau} \qquad \text{ where } \bar\tau:=\operatorname{Ad}(u)\circ\tau
 \]
is a connected subgroup of $K$ with maximal torus $S:=\exp \aaa=(T^\tau)^0$. Its root datum is $(\bar\Phi_\tau(a), \Lambda_\tau)$ where $\Phi_\tau(a):=\{\a\in\Phi:\alpha(a)=0\}$.
  
 \end{enumerate}

\end{thm}

\begin{defn}
We denote $S(a)=\{\a\in S: \a(a)=0\}$ for $a\in \bar\A$
\end{defn}

The alcove and the momentum image of a q-Hamiltonian manifold $M$ have the following relation: As $\bar\A$ is a fundamental domain for the Weyl group $W^\tau$, the map $\bar\A \to \mathfrak{t}/W^\tau$ is a homeomorphism. Having a moment map $m:M \to K$, we can introduce the invariant moment map $m_+:M\to\bar\A$ with the property that the it makes the following diagram commute:
\[
\begin{xy}
  \xymatrix{
      M \ar[r]^m \ar[d]_{m_+}    &   K \ar@{->>}[d]  \\
      \bar\A \ar[r]_{\text{bij}}             &   K/_\tau K   
  }
\end{xy}
\]

The bijection between the alcove and the set of twisted conjugacy classes was already known before, cf. \cite{MW04}. The image of $m_+$, called $\P_M=m_+(M)\subseteq \bar\A$ determines the actual image of $m$ via  $m(M)=K\cdot \exp(\P_M)$. It follows that $\P_M$ is a convex polytope lying in $ \bar\A$ and all fibers of $m_+$ (and $m$) are connected (\cite{Kno14}, 4.4).

Another important definition is the one of a local root system: Here, we use $\Phi_X:=\{\a\in\Phi:\a(X)=0\}$ for $X\in\P$.
\begin{defn}[\cite{Kno14}, definition 7.2.] 
A local system of roots (or local root system) $\Phi(*)$  on $\P$ is a family $(\Phi(X))_{X\in\P}$ of root systems on $\aaa$ such that for each $X\in\P$:
\begin{enumerate}
 \item $\Phi_Y=\Phi(X)_Y$ for all $Y$ in a sufficiently small neighborhood of $X$ in $\P$.
 \item Every root $\a\in\Phi(X)$ is either non-negative or non-positive on $\P$.
\end{enumerate}
An {\bf integral local root system on $\P$} is a pair $(\Phi(*), \Lambda)$ such that $(\Phi(x),\L)$ is an integral root system for every $x\in\P$.
\end{defn}
\begin{rem}\leavevmode\\
 We observe that setting $X=Y$ in the definition leads to $\Phi(X)=\{\a\in \Phi: \a(X)=0\}$.
\end{rem}

\part{Genuine Quasi-Hamiltonian Manifolds of Rank one}

This part is devoted to the classification of all compact genuine quasi-Hamiltonian $K$-manifolds for a simple and simply connected Lie group $K$ such that the corresponding spherical pair $(\P,\L_S)$ (ref. \cref{sphpair}) has the following properties:
\begin{itemize}
\item The moment polytope $\P$ is a line segment $[X_1X_2], X_i\in \bar\A \setminus \A$ that touches every wall of the fundamental alcove $\A$.
\item The lattice $\L_S$ is a rank-one-lattice and the generator $\w$ spans $P$ as a line.
\end{itemize}

For the rest of this part, we assume our polytopes and local models to have this property.

  \begin{prop}
Let $\N\w$ be the weight monoid of a local model for a quasi-Hamiltonian manifold $M$ of rank one in $X$. Then the moment polytope can be written as $\{X+\R\w\}\cap \A$. 
 \end{prop}
 
 \begin{proof}
 The cone spanned by $\w$ is the cone of the corresponding spherical variety via \cref{sphpair}. It is the line segment defined by $X+\R\w\cap \A$.
 \end{proof}

\section{Local models of rank one}\label{rank1model}

For rank reasons, it follows that for smooth affine spherical varieties $Z$ of rank one, only the two extreme cases for $Z:=G\times^H V$ can show up:
\begin{lem}\label{rank1var}
A smooth affine spherical $G$-variety $Z$ of rank one is of exactly one of the following types:
\begin{enumerate}
 \item $Z=G/H, V=0$. We will call this the ``homogeneous case'', as it means that $Z$ is a homogeneous spherical variety.
 \item $Z=V=\C^n, H=G=\SL(n,\C) \text{ or } V=\C^{2n}, H=G=\SP(2n, \C)$. We call this the ``inhomogeneous case''.
\end{enumerate}
\end{lem}

\begin{proof}
As the homogeneous space $G/H$ is the image of $Z=G\times^H V$ under the projection $Z\to G/H$, the rank of the homogeneous space $G/H$ is at most the rank of $Z$, so either $0$ or $1$.

If it is $0$, then $G/H$ is projective by \cite{Tim11}, Proposition 10.1, but, being also affine, it is a single point, i.e. $G=H$. We deduce $Z=V$, i.e. $Z$ is a spherical module of rank $1$. By the classification of spherical modules (cf. \cite{Kno98}), the module $V$ and the group $G$ must be of the form in statement (2) of the lemma.

Supose now $G/H$ has rank $1$, that means $Z$ and $G/H$ have the same rank. Let $K(a)\subseteq G$ be the stabilizer of a point $a$ in the open $G$-orbit of $Z$ such that $K\subseteq H$. By \cite{Gan10}, Lemma 2.4, the quotient $K/H$ is finite. This implies that the projection $Z\to G/H$ has finite fibers, i.e. $V$ contains only finitely many points. Hence $V=0$ and $Z$ is of the form in statement (1) of the lemma.
\end{proof}

Inhomogeneous smooth affine spherical varieties are are only possible for type $A$ or $C$.

\begin{lem}\label{inh1}
The generators of the weight monoids of smooth affine spherical homogeneous varieties are the $\w$ with the following values paired with the corresponding coroots of the drawn simple roots. We use the following diagrams for them:
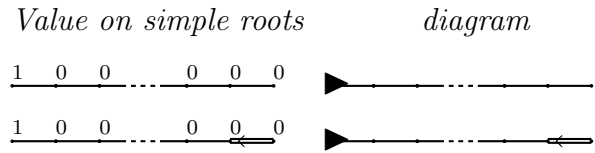
\begin{center}
\begin{tabular}{cc}
Value on simple roots & diagram\\
$\begin{picture}(12000,1000)
\put(0,0){\usebox\dynkinathree}
\put(3600,0){\usebox\susp}
\put(7200,0){\usebox\dynkinathree}
\put(0,300){\tiny $1$}
\multiput(1800,300)(1800,0){2}{\tiny $0$}
\multiput(7200,300)(1800,0){3}{\tiny $0$}
\end{picture}$&$\begin{picture}(12000,1000)
\put(0,0){\usebox\dynkinathree}
\put(3600,0){\usebox\susp}
\put(7200,0){\usebox\dynkinathree}
\put(-350,-300){$\btr$}
\end{picture}$\\
$\begin{picture}(12000,1000)
\put(0,0){\usebox\dynkinathree}
\put(3600,0){\usebox\susp}
\put(7200,0){\usebox\dynkincthree}
\put(0,300){\tiny $1$}
\multiput(1800,300)(1800,0){2}{\tiny $0$}
\multiput(7200,300)(1800,0){3}{\tiny $0$}
\end{picture}$&
$\begin{picture}(12000,1000)
\put(0,0){\usebox\dynkinathree}
\put(3600,0){\usebox\susp}
\put(7200,0){\usebox\dynkincthree}
\put(-350,-300){$\btr$}
\end{picture}$
\end{tabular}
\end{center}
\end{lem}

\begin{proof}
We already saw that inhomogeneous smooth affine spherical varieties of rank one are spherical modules. The classification of spherical modules in \cite{Kno98} gives the desired result.
\end{proof}

Note that there is a certain analogy to Luna's diagrams. Again, the simple roots in the support of $\w$ whose coroots pair zero with $\w$ are the ones in the diagram that are not marked. In that sense, the triangles in this diagrams have a similar meaning to the circles in Luna diagrams.

 \begin{defn}
 Let $X\in\bar\A$ and suppose $\Phi(X):=\{S\setminus\{\a_{n_1},\a_{n_2},\dots,\a_{n_m}\}\}:=S_{n_1, \dots, n_m}$  We denote the {\bf set of all homogeneous local models in $X$} by
\[
 H(X):=H(n_1, \dots, n_m)
\]
For example, by stating $\w\in H(0)$ we say that $\w$ is (the weight of) a homogeneous local model for $\Phi(X)=S \setminus \{\a_0\}$.\\

We also can define the {\bf set of inhomogeneous local models $w$ fulfilling $\la w, \aa_k^\vee \ra =1 $ in $X$}  as
\[
 I_k(X):=I_k(n_1, \dots, n_m) 
\]
\end{defn}
So, for example, $\w\in I_1(0)$ means that $\w$ is an inhomogeneous local model on $\Phi(X):= S \setminus \{\a_0\}$ and that $\w(\a_1):=\la \w, \aa_1^\vee\ra=1$. (cf. \cref{inh1})

Now let us look at some examples of moment polytopes of rank one.

 Let us assume that $\w$ is the weight of a {\bf homogeneous smooth affine spherical variety} of rank one in $X_1$ where $\P=[X_1X_2]$ is the moment polytope of a q-Hamiltonian manifold. Then $\w$ and the local model $\w^\sharp$ in $X_2$ must  generate the same lattice, hence $\w^\sharp=\w-t\delta$ for some $t\in \R$ and $\delta$ the root given by the Dynkin labels. We call that the case of a ``bi-homogeneous'' manifold (or momentum polytope).
 
 \begin{example}
 We fix the affine root system $A_n^{(1)}$. We take $X_1=(0,0,0,\dots,0)$, the local root system is $S(X_1)=S\setminus\{\alpha_0\}$. We think about the local model given by $\w=\aa_1+\aa_2+\dots+\aa_n=(1,0,\dots, 0, -1)$ which corresponds to the variety $\nicefrac{\SL(n+1)}{\GL(n)}$. We want our polytope to touch every wall of the alcove, hence we demand $\a_0 \in S(X_2)$. Then the other point where the moment polytope hits the walls of the alcove is $X_2=(1,0,\dots,0,-1)$ with local root system $\Phi\setminus\{\a_1, \a_n\}$. Because $\w^\sharp:=\w-\delta=-\aa_1-\aa_2-\dots-\aa_n=\aa_0$ is a spherical root of rank one in this local root system corresponding to $\nicefrac{\SL(2)}{T}$, we have found a bi-homogeneous variety. (Technically, the acting group is $\SL(2)\times \SL(n)$ with a trivial operation of the second factor. We will always leave out components that act trivially.)
\end{example}

\begin{example}
 We look at $A_n^{(1)}$ and again, $X_1=(0,0,0,\dots,0)$ with local root system $\Phi(X_1)=S\setminus\{\alpha_0\}$. The first fundamental weight $\w_1$ of $A_n$ is an inhomogeneous local model here, corresponding to $\SL(n+1)\times^{\SL(n+1)} \C^{n+1}$. The local root system of $X_2$ is $S\setminus\{\a_1\}$ (we shall have a general theorem for that later: \cref{inthm}) which also permits an inhomogeneous local model. We verify by elementary calculations that $-\w$ is an inhomogeneous local model there.
\end{example}

\begin{example}
 We consider $\operatorname{G}_2^{(1)}$. Then $\w=\aa_2+2\aa_1$ is a homogeneous model for $X_1=(0,0,0)$, the corresponding variety is $\nicefrac{\operatorname{G}_2}{\SL(3)}$. Hence $ S(X_2)=\{\a_0,\a_2\}$, and $-\w$ is an inhomogeneous model corresponding to $\SL(3)\times^{\SL(3)} \C^3$ there.
\end{example}

These three examples lead to the following definition:
\begin{defn}
Let  the line segment $[X_1X_2]=(X_1+\R\w)\cap\bar\A$ be the moment polytope of a quasi-Hamiltonian manifold. 
\begin{enumerate}
\item We call this moment polytope {\bf bi-homogeneous} if $\w$ is the generator of the weight monoid of a homogeneous smooth affine spherical variety in $X_1$ and $\w^\sharp :=t\d-\w$ for $t \in \R$ is the generator of the weight monoid of a homogeneous smooth affine spherical variety in $X_2$.
\item We call this moment polytope {\bf bi-inhomogeneous} if $\w$ is the generator of the weight monoid of an inhomogeneous smooth affine spherical variety in $X_1$ and $-\w$  is the generator of the weight monoid of an inhomogeneous smooth affine spherical variety in $X_2$.
\item We call this moment polytope {\bf mixed} if $\w$ is the generator of the weight monoid of a homogeneous smooth affine spherical variety in $X_1$ and $-\w$ is the generator of the weight monoid of an inhomogeneous smooth affine spherical variety in $X_2$.
\end{enumerate}
\end{defn}
Of course, part 3 of this definition is symmetric in switching $X_1$ and $X_2$.

\section{Structure of bi-homogeneous polytopes}

We first think about the structure of bi-homogeneous moment polytopes. Recall that $S(X_1):=\{\a\in S: \a(X_1)=0\}$. We want to prove: 

\begin{thm}\label{homthm} 
For all bi-homogeneous moment polytopes of rank one, $S(X_1)$ and $ S(X_2)$ contain at least $n-1$ roots.
\end{thm}

The rest of this section is to prove this statement.

\begin{defn}\leavevmode\\
\begin{enumerate}
 \item  Two simple roots are called neighbors if they are connected in the affine Dynkin diagram.
 \item  We call a simple root {\em at the end} of a diagram if it is only connected to one other root.
 \item If we say ``left'' or ``right'', we always mean left or right in the corresponding affine Dynkin diagram in \cite{Kac90}.
 \item  We call the root given by the Dynkin labels of an affine root system $\delta$, e.g. for $A_n^{(1)}$, we have $\d=\sum_{i=0}^n \a_i$.
 \item We say a simple root $\a_i$ is supported in $\w$ if its coefficient in $\w=\sum k_i \a_i$ is nonzero.
\end{enumerate}
\end{defn}

We remember that we asked our moment polytopes to ``touch every wall of the alcove'', which means every simple root must be included in $S(X_1)$ or $ S(X_2)$.

\begin{lem}\label{trick} 
Let $\alpha$ be a simple root, $\w$ a dominant weight of a smooth affine spherical variety of rank one that is a local model in $X_1$. Suppose $X_2=X_1+c\cdot \w$ for some $c \in \R$. Then:
\begin{enumerate}
\item  Let $\a \in \Phi(X_1)$ and $\a\in S^p(\w)$. Then $\a\in\Phi(X_2)$.
\item Let $\a \in \Phi(X_1)$ and $\alpha \notin S^p(\w)$. Then $\a \notin \Phi(X_2)$.
\item Let $\a \notin \Phi(X_1)$ and $\alpha \in S^p(\w)$. Then $\a \notin \Phi(X_2)$.
\item Let $\a \notin \Phi(X_1)$ and $\alpha \notin S^p(\w)$. Then $\a\in\Phi(X_2)$ if and only if $c=-\frac{\a(X_1)}{\la \w, \a^\vee\ra}$
\end{enumerate}
\end{lem}

\begin{proof}\leavevmode \\
We use the arithmetic of affine combinations, namely equation \eqref{afflin}.
\begin{enumerate}
\item Remember that $\a\in S^p(\w)$ means that $\la \w, \alpha ^\vee\ra=0$. Having $\a \in \Phi(X_1)$ means $\alpha(X_1)=0.$ Hence
 $
  \alpha(X_2)=\underbrace{\alpha(X_1)}_{=0}+c \cdot \underbrace {\la \w, \alpha^\vee\ra}_{=0}=0,
 $
so $\alpha$ is in $\Phi(X_2)$.
\item 
In this case, we have
$
 \a(X_2)=\underbrace{\alpha(X_1)}_{=0}+c\cdot\underbrace{\la \w, \alpha^\vee\ra}_{\ne 0}\ne 0
$, 
so $\a$ is not in $\Phi(X_2)$.
\item In this case, we have
$
 \a(X_2)=\underbrace{\alpha(X_1)}_{\ne 0}+c\cdot\underbrace{\la \w, \alpha^\vee\ra}_{=0}\ne 0
$, 
so $\a$ is not in $\Phi(X_2)$.
\item
In this case, we have
\[
 \a(X_2)=\underbrace{\alpha(X_1)}_{\ne 0}+c\cdot\underbrace{\la \w, \alpha^\vee\ra}_{\ne 0} 
\]
We have $\a \in \Phi(X_2)$ if and only if this equation is zero. Hence $c=-\frac{\a(X_1)}{\la \w, \a^\vee\ra}$.
\end{enumerate}
\end{proof}

\begin{rem}
 This means in particular: If $S(X_1)=S\setminus \{\a\}$, meaning there is exactly one simple root not in $S(X_1)$, we can always realize $\a\in S(X_2)$.
\end{rem}

Let us now discuss what $t\in\R$ can appear in $\w^\sharp=t\delta-\w$.
As spherical roots are linear combinations of simple roots and the coefficients in these linear combinations are in $\N[\frac12]$ and $\le 3$,  $t$ could only be chosen such that all coefficients of simple roots in $\delta$ are in $\N[\frac12]$ and no coefficient of $t\delta$ is bigger then twice the biggest coefficient of a spherical root possible in any subsystem of the affine root system (e.g. 6 in root systems with subsystems of type $B_3$ or $F$).  \\

\begin{lem}\label{homlem1}
Let $\A^0$ be the fundamental alcove of an affine root system, and $\w\in H(X_1)$. Then the roots not supported in $\w$ must all be contained in one component of connection of $ S(X_2)$, or they must form $A_1\times A_1$ in $ S(X_2)$.
\end{lem}
\begin{proof}
This is because the weights of homogeneous smooth affine spherical varieties are exactly the spherical roots of rank one. Inspecting \cref{lunad} gives the desired result.
\end{proof}

We shall use this to prove the following:

\begin{thm}
 Let $\Phi$ be an affine root system not of type $A_n^{(1)}$, $E_i^{(1)}, i\in \{6,7,8\}$. If $|S(X_1)|<n-1$, no bi-homogeneous polytopes of rank one are possible.
\end{thm}

\begin{proof}
We first state that for $D_4^{(3)}$ and $G_2^{(1)}$, the theorem is trivially true. So let us assume the local root system contains more than three simple roots.
  \begin{rem}
  The arguments in the following proof do not see the length of the several simple roots involved. So we only use the corresponding simply laced diagrams for instructive drawings. We mark the simple roots not in $(X_1)$ with a cross above.
 \end{rem}

We assume that three roots are missing in $S(X_1)$. The following things can happen:
\begin{itemize}
 \item The three missing roots are neighbors. In root systems that do not have an ending of type $D$, the middle root of the missing roots is always in $S^p$ for any spherical root we could choose, so it is never in $ S(X_2)$.\\
  
 If the local root system in $X_1$ had an ending of type $D$ at the (without loss of generality) left end and we had chosen $\{\a_{2}, \a_{1}, \a_0\}\notin S(X_1)$, then $\a_{1}$ and $\a_0 \in S^p\forall w$ on $S(X_1)$ and hence not in $ S(X_2)$ by \cref{trick}.
 
 If we had chosen $X_1$ such that $\{\a_{3},\a_{2},\a_{0}\}\notin S(X_1)$, we had $\a_0\in S^p$ for every $\w$ we could choose and hence $\a_0\notin  S(X_2)$.

 \item Two of the missing roots are neighbors at the, without loss of generality, right end of the root system, meaning $\a_{n-1}, \a_n$ for ending $B$ and $C$ and w.l.o.g $\a_{n-2}, \a_n$ for type $D$. For every $\w$ we could choose,  $\a_n \notin S(X_2)$ by  \cref{trick}.

 \item Two are neighbors,but not at the end of $S(X_1)$: 
 First assume the missing roots $\a_{n-2}, \a_{n-1}$ are surrounded by $A_1\times A_1$: 
\[ 
  \begin{picture}(2400,1800)(-300,-900)
		\put(-1800,0){\dots}        
        \put(0,0){\usebox\dynkinafour}
        \put(5400,0){\usebox\dynkinatwo}
        \put(-600,600){$\times$}
        \put(3000,600){$\times$}
        \put(2500,-1200){\tiny $n-2$}
        \put(4800, -1200){\tiny $n-1$}
        \put(4800,600){$\times$}
        \end{picture}
 \]

Then $\aa_{n-4},\aa_{n-2},\aa_{n-1}$ needed to be supported in $\d-t\w$ and are neither connected nor $A_1\times A_1$, as $\aa_{n-3}\notin  S(X_2)$, contradiction.\\

Assume now the two missing roots $\aa_{d},\aa_{d+1}$ are neighbored by something that is not $A_1\times A_1$. For a spherical root $\w$ supported on the component left of $\a_d$ we get $\a_{d+1}\in S^p$, so $\a_{d+1}$  will never be in $ S(X_2)$ by  \cref{trick}, same  the component right of $\a_{d+1}$ and $\a_d$.
 \[
  \begin{picture}(2400,1800)(-300,-900)
        \put(0,0){\usebox\dynkinafour}
        \put(5400,0){\usebox\dynkinatwo}
        \put(-1800,0){$\dots$}
        \put(8100,0){$\dots$}
        \put(3000,600){$\times$}
        \put(3000,-1200){\tiny $d$}
        \put(4800, -1200){\tiny $d+1$}
        \put(4800,600){$\times$}
        \end{picture}
 \]

 \item None of them are neighbors.\\
 If no ending of $\Phi$ is of type $D$, the three roots $\notin S(X_1)$ could cut out an $A_1\times A_1$:
  \[
  \begin{picture}(2400,1800)(-300,-900)
        \put(0,0){\usebox\dynkinafour}
        \put(5400,0){\usebox\dynkinatwo}
        \put(-600,600){$\times$}
        \put(3000,600){$\times$}
        \put(6600,600){$\times$}
        \put(1500,-900){$\tiny \a$}
        \put(5100, -900){$\tiny \a'$}
        \end{picture}
 \]
 If there are exactly five roots in $\Phi$, like in the instructive drawing above, the roots not supported in $\w=[\frac12]\a+\a'$ form an $A_1\times A_1\times A_1$, contradicting \cref{homlem1}.

Suppose now $\Phi$ has more then five simple roots and no ending of $S(X_1)$ is of type $D$:
   \[
  \begin{picture}(2400,1800)(-300,-900)
        \put(-1800,0){$\dots$}
        \put(9000,0){$\dots$}
        \multiput(0,0)(3600,0){3}{\usebox\dynkinatwo}
        \multiput(1800,0)(3600,0){2}{\usebox\shortsusp}
        \multiput(1200,900)(3600,0){3}{$\times$}
        \put(1200,-1200){$d$}
        \put(4800,-1200){$e$}
        \put(8400,-1200){$f$}
        \put(000,600){\textrm{I}}
        \put(3600,600){\textrm{II}}
        \put(7000,600){\textrm{III}}
        \put(10800,600){\textrm{IV}}
        \end{picture}
 \]
 We number the connected components from left to right. A $\w$ for component {\textrm I} or {\textrm I+II} or \textrm{II} has $\aa_f\in S^p$ and hence $\aa_f\notin  S(X_2)$. Analogous for an $\w$ for component {\textrm III, III+IV}  or {\textrm IV} and $\aa_d$. If component {\textrm I} and {\textrm IV} form $A_1\times A_1$, $\aa_e\notin  S(X_2)$ for $\w$ of type $D_2$. If component {\textrm I} and {\textrm II} form $A_1 \times A_1$, $\aa_f\notin  S(X_2)$ for such an $\w$. If component {\textrm I} and {\textrm III} form $A_1\times A_1$, the roots not supported in $\w$ of type $D_2$ supported on these two components are distributed over two connected components in $ S(X_2)$ that are not of type $A_1$, contradicting \cref{lunad}.\\
 
If the local root system had one ending of type $D$ and exactly 4 simple roots (which is the case for $B_3^{(1)}$ or $A_{3}^{(2)}$), we could have $\{\aa_0,\aa_1,\aa_3\}\notin S(X_1)$. But then these roots  have to be in $ S(X_2)$ and form three connected components. Impossible.\\
 
 If there is an ending of type $D$ on the w.l.o.g. left, it could happen that $S(X_1)$ only has two components. Let us assume $\a_0, \a_1, \a_d$ with $2<d<n-1$ are $\notin S(X_1)$.
 
 
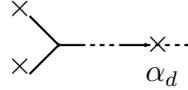
\[
\begin{picture}(0,2400)(0,0)
  \thicklines\put(1800,1200){\usebox\susp}
  \put(5400,1200){\usebox{\shortsusp}}
  \put(1800,1200){\line(-1,1){1200}}
  \put(1800,1200){\line(-1,-1){1200}}
  \put(-300,00){$\times$}
  \put(-300,2400){$\times$}
  \put(5400,900){$\times$}
  \put(5400,-300){$\a_d$}
 \end{picture}
\]
 
 One possible weight on the component right of $\a_d$ leads to $\a_0, \a_1\notin  S(X_2)$ by \cref{trick}.
The other possibility is something supporting $\{\a_2,\dots,\a_{d-1}\}$. But then $\delta-\w$ consists of three connected components of type either $A_1\times A_1 \times $something$ $ each containing one missing root if we had $\w=\aa_2+\dots+\aa_{d-1}$ or $A_3\times $something each containing a missing root if we had $\w=[\frac12]\aa_2+2\aa_3+\aa_4$ or $\w=[2]\aa_2$. Both oppose \cref{lunad}.\\

 If $\aa_0, \aa_1, \aa_n$ were $\notin S(X_1)$, the roots not supported in $\w$ were contained in two different components of $ S(X_2)$ that are not $A_1\times A_1$, contradicting\cref{homlem1}.\\

 Let us now assume that we are in $D_n^{(1)}$ and three roots at the ends are $\notin S(X_1)$, without loss of generality  $\a_0, \a_1, \a_n \notin S(X_1)$: 
  
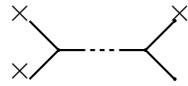
\[
\begin{picture}(0,3600)(0,0)
  \thicklines\put(1800,1200){\usebox\susp}
  \put(5400,1200){\usebox{\bifurc}}
  \put(1800,1200){\line(-1,1){1200}}
  \put(1800,1200){\line(-1,-1){1200}}
  \put(-300,0){$\times$}
  \put(-300,2400){$\times$}
  \put(6300,2400){$\times$}
 \end{picture}
\]
If we chose $\w=\aa_2+\dots+\aa_{n-1}$, we had $\aa_2, \aa_{n-1} \notin  S(X_2)$ via \cref{trick}, and hence the roots not supported in $\w$ are in three components of connection of $ S(X_2)$, contradicting \cref{homlem1}.
If we chose $[\frac12]\aa_2+2\aa_3+\aa_4$ in $D_5^{(1)}$, we had $ S(X_2)$ of type $A_3\times A_1 \times A_1$, where the missing roots are distributed over one $A_1$ and $A_3$, contradicting \cref{homlem1}.
\end{itemize}
This finishes the proof.
\end{proof}

Let us now settle the remaining cases. The following lemma is more general than needed to prove \cref{homthm} above.

\begin{lem}\label{homlem3}
 Let $\Phi$ be an affine root system of type $E_i^{(1)}, i \in \{6,7,8\}$. Then there are no bi-homogeneous manifolds of rank one.
\end{lem}

\begin{proof}
We remember that for a bi-homogeneous polytope generated by a weight $\w$, we have $\w+(\w^\sharp)=t\cdot \delta$ for some coefficient $t\in \R^\times$. As coefficients of simple roots in spherical roots are $\frac12,1,2,3$, $t$ must be chosen such that all coefficients of $t\delta$ are in $\N[\frac12]$ and $\le 4$. 

\begin{defn}
Let $k_(\aa_i)$ be the coefficient of $\aa_i$ in $\d$. For $\w$ the weight of a homogeneous variety of rank one, let $k_\w(\aa_i)$ be the coefficient of $\aa_i$ in $\w$, that is $\w=\sum k_\w(\aa_i) \cdot \aa_i$.
\end{defn}
We show the proof for $E_6^{(1)}$, the proof for the other cases is similar and can be found in \cite{Pau}.

For this case, the coefficients of simple roots in $\d$ are $1,2,3$. That means  $t\in\{\frac12, 1\}$.
 \begin{itemize}
  \item $t=\frac12$. Then $k(\aa_4)=\frac32$. As $E_6^{(1)}$ does not contain a subsystem of type $B_3$, it follows that w.l.o.g $k_\w(\aa_4)=1$ and $k_{\w^\sharp}(\aa_4)=\frac12$. That means that $\aa_4$ must be contained in $S(X_1)$ and $ S(X_2)$, meaning that $\aa_4$ is in $S^p$ by \cref{trick}. The remaining possibilities for $\w$ are the following (up to diagram automorphisms):
  \begin{itemize}

   \item $\w=\aa_1+2\aa_3+\aa_4$. 
   \item $\w=\aa_3+\aa_4+\aa_5$.
   \item $\w=\frac12(\aa_3+\aa_5+2\aa_4+2\aa_2)$
  \end{itemize}
  All are in contradiction to \cref{homlem1}
 \item $t=1$. Then $\w$ and $\w^\sharp$ must support $\aa_4$, w.l.o.g. $k_\w(\aa_4)=2$  and $k_{\w^\sharp}(\aa_4)=1$. This again forces $\aa_4\in S^p$. The only possibility for $\w$ is  $\w=\aa_3+\aa_5+2\aa_4+2\aa_2$ is in contradiction to \cref{homlem1}.
  \end{itemize}
So there are no bi-homogeneous polytopes for $E_6^{(1)}$.
\end{proof}

Finally, it remains to prove the theorem for $A_n^{(1)}$:

\begin{lem}\label{homlem5}
 We consider an affine root system of type $A_n^{(1)}$. If three or more roots are $\notin S(X_1)$, then there is no bi-homogeneous moment polytope. 
\end{lem}

\begin{proof} We distinguish the following cases:
 \begin{enumerate}
  \item Assume there are three missing roots that are neighbors, without loss of generality we assume $\aa_k, \aa_{k+1}, \aa_{k+2} \notin S(X_1)$.  But then $\aa_{k+1}\in S^p(\w), \aa_{k+1}(X_1)>0$. Hence $\aa_{k+1}\notin  S(X_2)$ by \cref{trick}.
  \item Assume that three roots are missing, two of the missing roots $\aa_k,  \aa_{k+1}$ are neighbors, and a third one, $\aa_l$ is not $\aa_{k-1}$ or $\aa_{k+2}$. Then we have to consider:
  \begin{itemize}
   \item $\w$ has $\aa_k+2, \dots, \aa_{l-1}$ in its support. But then $\aa_{k}\notin  S(X_2)$ because $\aa_k(X_1)\ne 0, \aa_k\in S^p$ and \cref{trick}.
   \item $\w$ has $\aa_{l+1}, \dots, \aa_{k-1}$ in its support. But then $\aa_{k+1}\notin  S(X_2)$ because $\aa_{k+1}(X_1) \ne 0,  \aa_{k+1}\in S^p$ and \cref{trick}.
  \end{itemize}
     \item Assume none of the three missing roots are neighbors (and hence $n\ge 5$), we call them $\aa_k, \aa_l, \aa_m$ clockwise. We assume $\w$ has $\aa_{k+1},\dots, \aa_{l-1}$ in its support. But then $\aa_m\notin  S(X_2)$ because $\aa_l(X_1)\ne 0,  \aa_l\in S^p$ and \cref{trick}.
     If the three roots $\notin S(X_1)$ are w.l.o.g. $\aa_0, \aa_{2}, \aa_{4}$ and $n\ge 5$, then the roots not in the support of $\w$ are in two components of support $A_1\times A_{\ge 3}$ in $ S(X_2)$, and there is no spherical root with this support via \cref{lunad}.
 \end{enumerate}
\end{proof}

This finishes the proof of \cref{homthm}.

\section{Structure of bi-inhomogeneous and mixed polytopes}

Fortunately, we can  find an even stronger necessary condition for the cases of bi-inhomogeneous and mixed polytopes.

Remember that we ask $\w$ to touch every wall of the alcove at least once.

We want to prove the following theorem which is an inhomogeneous analogon to \cref{trick}:
\begin{thm}[Structure of inhomogeneous polytopes]\label{inthm} 
 Let $\w$ be an inhomogeneous local model for a genuine quasi-Hamiltonian manifold of rank one for some $X_1 \in \bar \A_0$, and $\w\in I_k(S(X_1))$, meaning $\la \w, \aa_k^\vee \ra=1$.  Then:\\
 The local root system $ S(X_2):=\Phi(X_1+c\w)$ is 
 \[
  \Phi(X_2)=S \setminus \{\a_k\}
 \]
\end{thm}

\begin{proof}

Choose some subset $S(X_1)\subset S$ such that there is an inhomogeneous local model for $S(X_1)$. After renumbering, we can assume that $S(X_1)=\{\a_0, \dots, \a_p\}$ for some $p<n$.  We recognize that there are $n-p-1$ free coordinates in $\w$. We also know from \cref{inh1} that exactly one $\a_k(\w)>0$, we can assume $k=p$. So, $\a_i(\w)=0$ for $i=0, \dots, p-1$ and $\a_p(\w)>0$. We consider now the point $X_2=X_1+c\w$ for some positive real number $c$, the second point where $\P$ intersects the walls of the alcove. We have by  \cref{afflin}:
\[
 \a_i(X_2)=\a_i(X_1)+c\cdot \aa_i(\w)
\]
By construction, $\a_i(X_1)=0$ for $i\in\{0,...,p\}$. But $\aa_i(\w)$ is also zero for $i=0, \dots, p-1$. Furthermore, $\aa_p(\w)>0$. It follows:
\begin{equation}
 \a_0, \dots, \a_{p-1} \in  S(X_2), \a_p \notin  S(X_2)
\end{equation}

 We want that $\w$ touches every wall of the alcove at least once. It remains to prove that we can always find an $\w$ that leads to an $X_2$ such that all simple roots not in $S(X_1)$ are in $ S(X_2)$, that is that our $\P$ is genuine:   
 
 Recall that, by using the relations of coordinates given by $\a_0, \dots, \a_{p}$, we defined $p+1$ coordinates of $X_1$, $n-p-1$ coordinates of $X_1$ are still free. So we have $n-p-1+1=n-p$ free coordinates in $X_2$, because by going from $X_1$ to $X_2=X_1+c\w$, we get one additional parameter $c$. We have $n+1-p+1=n-p$ equations to solve, coming from the roots $\a_{p+1}, \dots, \a_n$. So, in general, we face an (in general inhomogeneous) system of $n-p$ equations in $n-p$ unknowns. Because the simple roots are linearly independent, this system has full rank and it is well known that it is always solvable with exactly one solution. So, it is always possible to have all roots $\a_{p+1}, \dots, \a_n$ in $ S(X_2)$. 
 Concluding, we have seen that $ S(X_2)$ is $S\setminus\{\a_k\}$: All simple roots supported in $\w$ w but $\a_k$ are still in $S(X_2)$, and all roots not in $S(X_1)$ are in $S(X_2)$.

\end{proof}

\begin{cor}[bi-inhomogeneous polytopes]\label{incor1} 
Bi-inhomogeneous polytopes are only possible if $S(X_1)=S\setminus\{\a_j\}$ and $ S(X_2):=S \setminus\{\a_k\}$ each contain $n$ roots. If $\w\in I_k(j)$, meaning $\la\w, \a_k^\vee\ra=1$ for $\a_k\in S(X_1)$, then $ S(X_2)=
S\setminus \{\a_k\}$ and there we must have $-\w\in I_j(k)$, and vice versa.
\end{cor}
\begin{proof}
 This is an immediate consequence of \cref{inthm}. Starting in without loss of generality $X_1$ with an inhomogeneous $\w$, $X_2$ is n-elemental and by the proof of the \cref{inthm}, $\a_k$ is not in $\Phi(X_2)$. As $\w$ takes the value zero on all roots of $ S(X_2)$ but $\a_j$, a bi-inhomogeneous polytope has only a chance to take value 1 there. 
\end{proof}
\begin{rem}
This means: if $\P=[X_1X_2]$ is a bi-inhomogeneous momentum polytope, then $X_1$ and $X_2$ are vertices of $\bar \A$.
\end{rem}

It also follows: in case of a bi-inhomogeneous polytope, $\w$ is completely determined by the root $\a_k$ with $\la \w, \aa_k^\vee\ra=1$ and the one simple root not contained in the local root system $S(X_1)$. There are no free parameters! Hence instead of writing $\w\in I_k(j)$, we can and will write $\w=I_k(j)$.

The fact $|S(X_1)|=| S(X_2)|=n$ leads to a strong tool for finding bi-inhomogeneous polytopes: It is necessary that the root system has two n-elemental subsets that allow the existence of inhomogeneous local models,  meaning that there are components supporting $n$ roots and have one connected component of type $A_l$ or $C_l$. We also know from \cref{incor1} which roots have to be or cannot be in the one or the other. That significantly reduces the number of cases remaining to check. 

We also find a necessary condition for mixed polytopes:
\begin{cor}[mixed momentum polytopes]\label{mixcor2} 
 Mixed momentum polytopes (homogeneous in $X_1$, inhomogeneous in $X_2$) can only exist if $|S(X_1)|=n$.
\end{cor}
\begin{proof}
 Let $\w$ be an inhomogeneous local model in $X_2$. Then $|S(X_1)|=n$ by \cref{inthm}.
\end{proof}

\section{Explicit determination of moment polytopes of rank one}

We first sum up the last three sections:

\begin{thm}\label{qhamnot}
For every bi-homogeneous moment polytope $[X_1X_2]$ of any compact genuine quasi-Hamiltonian $K$-manifold of rank one, $S(X_1)$ and $ S(X_2)$ contain at least $n-1$ roots ($n$ denotes the rank of $K$). For every mixed momentum polytope with homogeneous local model in $X_1$, the local root system $S(X_1)$ contains $n$ roots. For every bi-homogeneous moment polytope, we have $|S(X_1)|=|S(X_2)|=n$.
\end{thm}

\begin{proof}
This is the combination of \cref{homthm}, \cref{incor1} and \cref{mixcor2}.
\end{proof}

In general, for a fixed local root system $\Phi_X$, there are only finitely many local models possible, namely the homogeneous spherical varieties found in \cite{Akh83} or the spherical modules of rank one for $A$ and $C$. So which local models are candidates in one vertex of the alcove is determined by the local root systems and the known classification of smooth affine spherical varieties, we discussed this in detail in \cref{rank1model}.

We shall establish a short notation for going through our algorithm. We will use the following tables:

\begin{center}
\begin{longtable}{|c|c|c|c|c|}
$S(X_1)$ & choice of $\w$ & $ S(X_2)$ & $\w^\sharp$ hom.? & $-w$ inhom.?\\\endhead
\end{longtable}
\end{center}

In the first column, we choose the subsystem we shall start with as a local root system in one ``vertex'' of our polytope by denoting its simple roots. That corresponds to some wall of the alcove. Then we denote all possible choices we have for the local model there in column two where we write $I_k$ for an inhomogeneous model being $1$ on the root $\a_k$ and zero elsewhere. The next column gives the local root system at the other end of our polytope, in $X_2$ which is easy to see via \cref{trick}, \cref{incor1} or elementary calculations.

We will not explicitly mention cases that only differ by a diagram automorphism.

As the case-by-case-studies are pretty similar for all affine root systems, we shall only discuss two instructive cases in detail. We refer to \cite{Pau} for a detailed study of all affine root systems.

\subsection{$A_n^{(1)}$}
Note that in $A_n^{(1)}$, all subsystems of $n$ roots only differ by a diagram automorphism. We will choose $S\setminus\{a_0\}$ as a representative for our examination.

As this root system has the structure of a circle, all sums below have to be read ``modulo $n+1$''.

\begin{center}
\tiny{
\begin{longtable}{|p{2.5cm}|p{3cm}|p{2cm}|p{3cm}|p{3cm}|}
\hline
$S(X_1)$ & choice of $\w$ & $ S(X_2)$ & $\w^\sharp$ hom.? & $-w$ inhom.?\\\endhead
\hline
\multicolumn{5}{|c|}{$n=1$} \\
\hline
$\a_1$  & $I_1$ & $\a_0$ &  & $-\w=I_0$\\
\hline
 & $\aa_1$ & $\a_0$ & $\w^\sharp=\a_0$ & \\
 \hline
 & $2\aa_1$ & $\a_0$ & $\w^\sharp=2\aa_0$ & \\
\hline
\multicolumn{5}{|c|}{$n\ge 2$} \\
\hline
$S\setminus \{\a_0\}$  & $\a_{1,n}$ & $S\setminus\{\a_1, \a_n\}$ &$\w^\sharp=\aa_0$  & \\
\hline
  & $I_1$ & $S\setminus\{\a_1\}$ & &$-\w=I_0$   \\
\hline
$S\setminus \{\a_0\}, n=3$  & $[\frac12]\aa_1+2\aa_2+\aa_3$ & $S\setminus\{\a_2\}$ &$\w^\sharp=[\frac 12]\aa_3+2\aa_0+\aa_1$  & \\
\hline
$S\setminus\{\aa_1,\aa_n\}$ & $2\aa_0$ & $S\setminus\{\a_0\}$ & no: $\w^\sharp=2\a_{1,n}$, \cref{lunad} & no: $\la -\w, \aa_1^\vee\ra = \la-\w, \aa_n^\vee\ra>0$, \cref{inh1}\\
\hline
$S\setminus \{\a_d, \a_e\};\linebreak d\le e$ & $\aa_{d+1,e-1}$ & $S\setminus \{\a_{d+1}, \a_{e-1}\}$& $\w^\sharp=\a_{e,d}$ & \\
\hline
$S\setminus\{\a_0, \a_4\}$, $n\ge 4$ & $[\frac 12] \aa_1+2\aa_2+\aa_3$ & $S \setminus \{\a_2\}$ & no:  $\w^\sharp=[\frac 12] 2\aa_0+\aa_1+\aa_3+2\aa_4+\dots+2\aa_n$, \cref{lunad} & no:$\la-\w,\a_3^\vee\ra=0,\linebreak \la -\w, \aa_4^\vee\ra=[\frac 12]1$, \cref{inh1}\\
\hline
$S\setminus\{\a_1,\a_3\}, n=3$ & $[\frac 12] \aa_1+\aa_3$ & $\a_2, \a_0$ & $\w^\sharp=[\frac 12] \aa_2+\aa_0$ & \\
\hline
$S\setminus\{\a_1,\a_3\}, n>3$ & $[\frac 12] \aa_1+\aa_3$ & $\ssm{1}{3}$ & no: $\w^\sharp=[\frac 12] \aa_0+\aa_{2,n}$, \cref{lunad} & no: $\Sp{2}>0, \Sp{3}>0$, \cref{inh1} \\
\hline
\end{longtable}
}
\end{center}

\subsection{$D_n^{(1)}, n\ge 4$}

We first argue that there are no bi-inhomogeneous polytopes.

\begin{itemize}
 \item $S(X_1)=S\setminus\{\a_0\}$ or another root at the end. Then $S(X_1)$ is of type $D$, where no inhomogeneous model is possible via \cref{inh1}.
 \item $S(X_1)=S\setminus\{\a_2\}$. Then $\w=I_0$ (or $w=I_1$,or if $n=4$: $\w=I_3$ or $\w=I_4$) is possible. But then, $ S(X_2)=S\setminus\{\a_0\}$ is of type $D$.
 \item $S(X_1)=S\setminus\{\a_d\}, 2<d<n-1$. But then $S(X_1)$ is of type $D_l\times D_k$ for some $l,k$.
\end{itemize}

Now let us continue for bi-homogeneous and mixed polytopes.
\begin{center}
\tiny{
\begin{longtable}{|p{2.5cm}|p{3cm}|p{2cm}|p{3cm}|p{3cm}|}
\hline
$S(X_1)$ & choice of $\w$ & $ S(X_2)$ & $\w^\sharp$ hom.? & $-w$ inhom.?\\\endhead
\hline
\multicolumn{5}{|c|}{$n=4$} \\
\hline
$\Ssm\{\a_0\}$ & $[\frac12]2\aa_3+2\aa_2+\aa_4+\aa_1$ & $\Ssm\{\a_3\}$ & $\w^\sharp=[\frac12]2\aa_0+2\aa_2+\aa_1+\aa_3$ & \\
\hline
$\Ssm\{\a_0,\a_1\}$ & $ [\frac12]\aa_4+2\aa_2+\aa_3$ & $\Ssm\{\a_2\}$ & $\w^\sharp=[\frac12]\aa_0+\aa_1$ & \\
\hline
& $\aa_4+\aa_2+\aa_4$ & $\Ssm\{\a_3,\a_4\}$ & $\w^\sharp=\aa_0+\aa_2+\aa_1$ & \\
\hline
$\Ssm\{\a_0,\a_2\}$ &$[2]\aa_1$ & elementary: never contains both missing roots & & \\
 \hline
$\Ssm\{\a_0,\a_2\}$ &$[\frac12]\aa_1+\aa_4$ & elementary: never contains both missing roots & & \\
\hline
$\Ssm\{\a_2\}$ & $[\frac12]\aa_0+\aa_1$ & $\Ssm\{\a_0,\a_2\}$ & already seen: $\w^\sharp=[\frac12]\aa_4+2\aa_2+\aa_3$ & \\
\hline
 & $[2]\aa_1$ & $\Ssm\{\a_1\}$ & no: $\w^\sharp=[2]\aa_0+2\aa_2+\aa_3+\aa_4$, \cref{lunad} & no: System of type $D$, \cref{inh1}\\
 \hline
\multicolumn{5}{|c|}{$n\ge 5$} \\
\hline
$\Ssm\{\a_0\}$ & $[\frac12]2\aa_{1,n-2}+\aa_{n-1}+\aa_n$ & $\Ssm\{\a_1\}$ & $\w^\sharp=[\frac12]2\aa_0+2\aa_{2,n-2}+\aa_{n-1}+\aa_n$ & \\
\hline
$\Ssm\{\a_0,\a_1\}$ & $[\frac12]2\aa_{2, n-2}+\aa_{n-1}+\aa_n$ & $\Ssm\{\a_2\}$ & $\w^\sharp=[\frac12]\aa_0+\aa_2$ & \\
\hline
$\Ssm\{\a_0,\a_2\}$ & $[2]\aa_1$ & $\a_0\notin  S(X_2)$, \cref{trick} & & \\
\hline
& $\w \in C(n)$ & $\a_0\notin  S(X_2)$, \cref{trick} & & \\
\hline

$\Ssm\{\a_0, \a_d\}, 2<d<n-1$ & $\aa_{1,d-1}$ & $\Ssm\{\a_1,\a_{d-1}\}$ & no: $\w^\sharp=\aa_0+2\aa_{2,n-2}+\aa_{n-1}+\aa_n$, \cref{lunad} &no: $\Sp{0}>0, \Sp{d}>0$, \cref{mixcor2}\\
\hline
$d=4$ & $[\frac12]\aa_1+2\aa_2+\aa_3$ & $\Ssm\{\a_2\}$ & no: $\S(X_2)$ neither connected nor $A_1\times A_1$, \cref{homlem1} & no: $\Sp{0}>0,\Sp{4}>0$, \cref{mixcor2}\\
\hline
& $\w \in C(n)$ & $\a_0\notin S(X_2)$ \cref{trick} & & \\
\hline

$\Ssm\{\a_0,\a_n\}$ & $\aa_{1,n-1}$ & $\Ssm\{\a_1,\a_{n-1}\}$ & $\w^\sharp=\aa_0+\aa_{2,n-2}+\aa_{n}$ & \\
\hline

$S\setminus\{\a_d\}$ with w.l.o.g $2\le d \le n/2$ & $[\frac12]\aa_0+\aa_1+2\aa_{2,d-1}$ & $\Ssm\{\a_{d-1}\}$ & $\w^\sharp=[\frac12]2\aa_{d,n-2}+\aa_{n-1}+\aa_n$ & \\

\hline$\d-\w$ neither connected nor $A_1\times A_1$
& $[\frac12]2\aa_{d+1,n-2}+\aa_{n-1}+\aa_n$ & $\Ssm\{\a_{d+1}$ & $\w^\sharp=[\frac12]\aa_0+\aa_1+2\aa_{2,d}$ & \\
\hline
$d=2$ & $[2]\aa_1$ & $\Ssm\{\a_1\}$ & no: $\w^\sharp=[2]\aa_0+2\aa_{2,n-2}+\aa_{n-1}+\aa_n$, \cref{lunad} & no: $\la-\w,\aa_3^\vee\ra=1$, \cref{inh1}\\
\hline

$\ssm{d}{e}, 1<d,d+1< e<n-1$ & $\w\in C(0)$ & $\a_e\notin S(X_2)$& &\\
\hline
& $\w\in C(n)$ & $\a_d\notin S(X_2)$ \cref{trick}& &\\
\hline
&$\w\in C(d+1)$ & & no: \cref{homlem1} & no: \cref{mixcor2} \\
\hline 
$\ssm{d}{d+1}, 1<d<n-2$ & $\w\in C(0)$ & $\a_{d+1}\notin S(X_2)$, \cref{trick}& &\\
\hline
& $\w\in C(n)$ & $\a_d\notin  S(X_2)$, \cref{trick} & &\\
\hline

$\ssm{2}{n}$ & $\w\in C(0)$ or $\w\in C(0)+C(1)$ & $\a_n\notin  S(X_2)$ & & \\
\hline
& $\w\in C(3)$ &  & no: $\w^\sharp$ not $\in$ \cref{lunad} & no: $\Sp{2}>0, \Sp{n}<0$, \cref{mixcor2} \\
\hline
\end{longtable}
}
\end{center}

\section{List of genuine Moment Polytopes of Rank one}

\begin{thm}[Momentum Polytopes of Rank one] 
 A compact multiplicity free genuine quasi-Hamiltonian $K$-manifold of rank one for $K$ simple corresponds to a spherical pair $(\P,\L_S)$ with the following properties:
 \begin{itemize}
\item $\P$ is a line segment $[X_1X_2]$ that is spanned as a line by $X_1+\R\w$ for some $\w$ found in \cref{ListHom} and $X_1$ such that $S(X_1)$ is the corresponding local root system.
\item We have $\a(X_1)=0$ or $\a(X_2)=0$ for every simple root $\a_0, \a_1, \dots, \a_n$
\item The lattice $\L_S$ is generated by this $\w$
\end{itemize}

\end{thm}

All parameters are assumed to be non-negative integers here. We mark inhomogeneous local models by printing them in {\bf bold font}. A $[2]$ in front of a momentum polytope says that not only $\w$ and $\w^\sharp$, but also $2\w$ and $2 \w^\sharp$ is possible. In the "local model"-columns we list the primitive spherical triple, containing the information about the semisimple type of the model. Note that we will only note the components of the local model with non-trivial operation.\\

We also investigated the operation of the connected center that could appear in the cases where the diagram of the primitive triples shows an asterisk. If these $\tt^1$ acted via a scalar $N$, we put an $\C^k_N$ in the $V$-entry of the triple. We omit this if $N=1$ which stands for the "standard" operation.
To find this scalar $N$, we determine a linearly independent set of primitive $\zeta\in\Z\S$ corresponding to the center of the Levi of the local model. It is defined by the fact that its elements are perpendicular to all simple roots in the local root system. Then it is well known that $N=\la \w, \zeta^\vee\ra$.\\
Remarkably, it turned out that $N=1$ for all genuine quasi-Hamiltonian Manifolds of rank one. We shall see that this is not true for those quasi-Hamiltonians of rank one that are not genuine in the next part.

\pagebreak

\begin{lis}\label{ListHom} 
\begin{center}
List of quasi-Hamiltonian manifolds of rank one
\tiny{
\begin{longtable}{|p{0.7cm}|p{2cm}|p{3.3cm}|p{2cm}|p{3.3cm}||p{2cm}|}
\hline
&$\w$ &local model & $\w^\sharp$ &local model & remarks\\\endhead
\hline
\multicolumn{6}{|c|}{$A_1^{(1)}$} \\
\hline
A1)&$2\aa_0 \in H(1)$ & $(\sll(2),0,0)$ & $2\aa_1\in H(0)$ &$(\sll(2),0,0)$ & \\
\hline
\multicolumn{6}{|c|}{$A_3^{(1)}$} \\
\hline
A2)&$[\frac12](\aa_1+\aa_3) \in H(0,2)$&$(\sll(2)\times\sll(2),\Delta\sll(2),0)$ & $[\frac 12](\aa_0+\aa_2) \in H(1,3)$&$(\sll(2)\times\sll(2),\Delta\sll(2),0)$ & \\
\hline
A3)&$[2]\frac12(\aa_i+2\aa_{i+1}+\aa_{i+2}) \in H(0)$ & $(\sll(4),\sp(4),0)$ & $[2]\frac12(\aa_{i+2}+2\aa_{i+3}+\aa_{i})\in H(2)$&$(\sll(4),\sp(4),0)$& $i\in\{0,1,2,3\}$, everything $\mod 4$ \\
\hline
\multicolumn{6}{|c|}{$A_n^{(1)}, n\ge 1$} \\
\hline
A4)&$\aa_d+\dots+\aa_e \in H(d-1,e+1)$ & $(\sll(e-d+1),\tt^1+\sll(e-d),0) $  & $\aa_{e+1}+ \dots+\aa_{d-1} \in H(d,e)$ & $(\sll(n-(e-d)),\tt^1+\sll(n-(e-d+1)),0)$ & $0\le d \le e \le n$, everything  $\mod n+1$. \\
\hline
A5)&$\boldsymbol{\w = I_k(k+1)}$ & $(\sll(n+1),\sll(n+1),\C^{n+1})$ & $\boldsymbol{-\w = I_{k+1}(k)}$ &$(\sll(n+1),\sll(n+1),\C^{n+1})$ & $0\le k \le n$, everything $\mod n+1$.\\
\hline
\multicolumn{6}{|c|}{$B_3^{(1)}$} \\
\hline
B1)&$\frac12 (\aa_0+2\aa_2+3\aa_3) \in H(1)$ &$(\soo(7),G_2,0)$& $\boldsymbol{-\w \in I_1(3)}$ &$(\sll(4),\sll(4),\C^4) $& or 0 and 1 switched \\
\hline
B2)&$\frac 12 (\aa_1+\aa_3)\in H(2)$ &$(\sll(2)\times \sll(2),\Delta\sll(2),0)$& $\boldsymbol{-\w \in I_2(1,3)}$ &$(\sll(3),\sll(3),\C^3)$ & or 0 and 1 switched \\
\hline

\multicolumn{6}{|c|}{$B_n^{(1)}$} \\
\hline
B3)&$[2](\aa_1+\dots+\aa_n) \in H(0)$ &$(\soo(2n+1),\soo(2n),0)$ &$[2](\aa_0+\aa_2+ \dots+\aa_n) \in H(1)$ &$(\soo(2n+1),\soo(2n),0)$ &  \\
\hline
B4)&$[2](\frac 12(\aa_0+\aa_1)+\aa_2+\dots+\aa_d)\in \break H(d+1)$ &$(\soo(2k+2),\soo(2k+1),0)$ & $[2](\aa_{d+1}+\dots+\aa_n)\in H(d)$&$(\soo(2n-2k+1),\soo(2n-2k),0)$& $1\le d < n$\\
\hline
\addlinespace[0.3cm]
\hline
\multicolumn{6}{|c|}{$C_2^{(1)}$} \\
\hline
C1)&$[2]\frac12(\aa_0+\aa_2) \in H(1)$ &$(\sll(2)\times \sll(2),\Delta\sll(2),0)$& $[2]\aa_{1} \in H(0,2)$ &$(\sll(2),0,0)$&  \\
\hline
C2)&$[2](\aa_0+\aa_1) \in H(2)$ &$(\soo(5),\soo(4),0)$& $[2](\aa_{1}+\aa_2) \in H(0)$ &$(\soo(5),\soo(4),0$&  \\
\hline
\multicolumn{6}{|c|}{$C_n^{(1)}$} \\
\hline
C3)&$\frac 12 (\aa_0+\aa_n) \in H(1, n-1)$ &$(\sll(2)\times \sll(2),\Delta\sll(2),0)$ & $\aa_1+\dots+\aa_{n-1} \in H(0,n)$ & $(\sll(n-1),\tt^1+\sll(n-2),0)$&see also $n=2$ \\
\hline
C4)&$\aa_0+2\aa_1+\dots+2\aa_{d-1}+\aa_d\in H(d+1)$ &$(\sp(2d+2),(\sll(2)+\sp(2d)),0)$& $\aa_{d}+2\aa_{d+1}+\dots+2\aa_{n-1}+\aa_n \in H(d-1)$ &$(\sp(2n-2d+2),(\sll(2)+\sp(2n-2d)),0)$&$1 \le d <n$ \\
\hline
C5)&$\boldsymbol{\w= I_k-1(k)}$ &$(\sp(2k),\sp(2k),\C^{2k})$& $\boldsymbol{\w = I_{k}(k-1)}$&$(\sp(2n-2k+2),\sp(2n-2k+2),\C^{2n-2k+2})$ &$1 \le k \le n$ \\
\hline
\pagebreak
\hline
\multicolumn{6}{|c|}{$D_4^{(1)}$} \\
\hline
D1)&$[2] \aa_i+\aa_2+\frac 12 \aa_j +\frac 12 \aa_k \in H(l)$ &$(\soo(8),\soo(7),0)$& [2] $\aa_l+\aa_2+\frac 12\aa_j+\frac 12 \aa_k\in H(i)$ &$(\soo(8),\soo(7),0)$& $i \ne j\ne k \ne l$;\break $i,j,k,l\in 0,1,3,4$  \\
\hline
D2)&$[1/2]\aa_i+2\aa_2+\aa_j \in H(k,l)$&$(\sll(4),\sp(4),0)$ &$[1/2] \aa_k+\aa_l\in H(2)$&$(\sll(4),\sp(4),0)$& $i,j,k,l \in 0,1,3,4$\break $i\ne j\ne k \ne l$  \\
\hline
D3)&$\aa_i+\aa_2+\aa_j \in H(k,l)$ &$(\sll(4),\tt^1+\sll(3),0)$ &$\aa_k +\aa_2+\aa_l\in H(i,j)$ &$(\sll(4),\tt^1+\sll(3),0)$ & $i,j,k,l \in 0,1,3,4$\break $i\ne j\ne k \ne l$ \\
\hline
\multicolumn{6}{|c|}{$D_n^{(1)}$} \\
\hline
&$\w$ &local model & $\w^\sharp$ &local model & remarks\\\hline
D4)&$[2](\aa_1+\aa_2+\dots+\aa_{n-2}+\frac12 \aa_{n-1}+\frac12{\aa_n}) \in H(0)$ &$(\soo(2n),\soo(2n-1),0)$ & $[2](\aa_0+\aa_2+\dots+\aa_{n-2}+\frac 12\aa_{n-1}+\frac 12\aa_n)  \in H(1)$ &$(\soo(2n),\soo(2n-1),0)$& \\
\hline
D5)&$[2](\aa_{n-1}+\aa_{n-2}+\dots+\aa_2+\frac12 \aa_0+\frac 12 \aa_1) \in H(n)$ &$(\soo(2n),\soo(2n-1),0)$& $[2](\aa_n+\aa_{n-2}+\dots+\aa_2+\frac 12 \aa_1 +\frac 12 \aa_0)\in H(n-1)$ &$(\soo(2n),\soo(2n-1),0)$& \\
\hline
D6)&$[2](\frac 12\aa_0+\frac 12 \aa_1+\aa_2+\dots+\aa_{d}) \in H(d+1)$ &$(\soo(2d+2),\soo(2d+1),0)$& $[2](\aa_{d+1}+\dots+\aa_{n-2}+\frac 12\aa_{n-1}+\frac 12 \aa_n)\break  \in H(d)$ &$(\soo(2n-2d),\soo(2n-2d-1),0)$& $1\le d \le n-2$   \\
\hline
D7)&$\aa_1+ \dots+\aa_{n-1} \in H(0,n)$ &$(\sll(n),\tt^1+\sll(n-1),0)$ & $\aa_{0}+\aa_2+\dots+\aa_{n-2}+ \aa_n \break \in H(1,n-1)$ &$(\sll(n),\tt^1+\sll(n-1),0)$&  or $\a_n$ and $\a_{n-1}$ switched  \\
\hline
\addlinespace[0.3cm]
\hline
\multicolumn{6}{|c|}{$F_4^{(1)}$} \\
\hline
F1)&$\aa_1+2\aa_2+3\aa_3+2\aa_4\in H(0)$ &$(F_4,\soo(9),0)$& $\aa_0+\aa_1+\aa_2+\aa_3 \in H(4)$&$(\soo(9),\soo(8),0)$& \\
\hline
F2)&$\aa_2+2\aa_3+\aa_4 \in H(1)$ &$(\sll(4),\sp(4),0)$& $\frac 12 \aa_o+\aa_1+\frac 12 \aa_2 \in H(3)$ &$(\sll(4),\sp(4),0)$& \\
\hline
F3)&{$\aa_3+\aa_4\in H(2)$} &$(\sll(3),\gll(2),0)$&$\boldsymbol{-\w \in I_2(3,4)}$&$(\gll(4),\gll(4),\C^4)$&\\
\hline
\addlinespace[0.3cm]
\hline
\multicolumn{6}{|c|}{$G_2^{(1)}$} \\
\hline
&$\w$ &local model & $\w^\sharp$ &local model & remarks\\\hline
G1)&{$\aa_2+2\aa_1 \in H(0)$} &$(G_2,\sll(3),0)$& $\boldsymbol{-\w = I_0(1)}$&$(\sll(3),\sll(3),\C^3)$ &  \\
\hline
G2)&{$\aa_1\in H(2)$}&$(\sll(2),0,0)$ & $\boldsymbol{-\w= I_2(1)}$ &$(\sll(3),\sll(3),\C^3)$& \\
\hline
G3)&{$\frac 12(\aa_0+\aa_1) \in H(2)$}&$(\sll(2)\times \sll(2),\Delta\sll(2),0)$& $\boldsymbol{-\w\in I_2(0,1)}$ &$(\gll(2),\gll(2),\C^2)$& \\
\hline
\addlinespace[0.3cm]
\hline
\multicolumn{6}{|c|}{$A_2^{(2)}$} \\
\hline
A6)&$2\aa_0 \in H(1)$ &$(\sll(2),0,0)$& $\aa_1\in H(0)$ &$(\sll(2),0,0)$& \\
\hline
A7)&{$\a_0 \in H(1)$} &$(\sll(2),0,0)$& $\boldsymbol{-\w= I_1(0)}$ &$(\sll(2),\sll(2),\C^2)$&\\
\hline
\pagebreak
\hline
\multicolumn{6}{|c|}{$A_{2n}^{(2)}$} \\
\hline
A8)&$2\aa_0+2\aa_1+\dots+2\aa_{n-1} \in H(n)$ &$(\soo(2n+1),\soo(2n),0)$& $\aa_n\in H(n-1)$&$(\sll(2),0,0)$& \\
\hline
A9)&{$\a_0+\aa_1+\dots+\aa_d \in H(d+1)$} &$(\soo(2(d+1)+1),\soo(2(d+1)),0)$&$\boldsymbol{-\w= I_{d+1}(d)}$&$(\sp(2n-2d),\sp(2n-2d),\C^{2n-2d})$ &$0\le d < n$\\
\hline
\addlinespace[0.3cm]
\hline
\multicolumn{6}{|c|}{$A_{2n-1}^{(2)}$} \\
\hline
{\tiny A10)}&$\aa_1+2\aa_2+\dots+2\aa_{n-1}+\aa_n \in H(0)$ &$(\sp(2n),(\sll(2)+\sp(2n-2)),0)$& $\aa_0\in H(2)$&$(\sll(2),0,0)$ &or $\aa_0$ and $\aa_1$ switched \\
\hline
{\tiny A11)}&$\aa_0+\aa_1+2\aa_2+\dots+2\aa_{n-1}\in H(n)$ &$(\soo(2n),\soo(2n-1),0)$& $\aa_n\in H(n-1)$ &$(\sll(2),0,0)$& \\
\hline
{\tiny A12)}&$\aa_0+\aa_1+\aa_2 \in H(3)$ &$(\sll(4),\gll(3),0)$& $\aa_2+2\aa_3+\dots +2\aa_{n-1}+\aa_n \in H(0,1)$ &$(\sp(2(n-1),(\sll(2)+\sp(2(n-2))),0)$& \\
\hline 
{\tiny A13)}&$\boldsymbol{\w = I_1(0)}$ &$(\sp(2n),\sp(2n),\C^{2n})$&$\boldsymbol{-w= I_0(1)}$ &$(\sp(2n),\sp(2n),\C^{2n})$& \\
\hline
{\tiny A14)}&{$\frac  12 \aa_0+\frac 12 \aa_1+\aa_2+\dots+\aa_d\in H(d+1)$}&$(\soo(2(d+1)),\soo(2(d+1)-1),0)$&$\boldsymbol{-\w = I_{d+1}(d)}$ &$(\sp(2n-2d),\sp(2n-2d),\C^{2n-2d})$&$1\le d \le n$\\
\hline
\multicolumn{6}{|c|}{$D_{3}^{(2)}$} \\
\hline
D8)&$\frac{1}{2}(\aa_0+\aa_2)\in H(1)$ &$(\sll(2)\times \sll(2),\Delta\sll(2),0)$& $\boldsymbol {-\w \in I_1(0,2)}$ &$(\gll(2),\gll(2),\C^2)$ & \\
\hline
D9)&$\boldsymbol{\w\in I_0(2)}$ &$(\sll(3),\sll(3),C^3)$& $\boldsymbol{-\w\in I_2(0)}$     &$(\sll(3),\sll(3),\C^3)$&   \\
\hline
\multicolumn{6}{|c|}{$D_4^{(2)}$} \\
\hline
{\tiny D10)}&$[2]\frac12(\aa_0+\aa_2)\in H(1,3)$ &$(\sll(2)\times \sll(2),\Delta\sll(2),0)$& $[2]\frac12(\aa_1+\aa_3)\in H(0,2)$     &$(\sll(2)\times \sll(2),\Delta\sll(2),0)$&   \\
\hline
{\tiny D11)}&$[2]\frac12(3\aa_0+2\aa_1+\aa_2) \in H(3)$ &$(\soo(7),G_2,0)$& $[2]\frac12(\aa_1+2\aa_2+3\aa_3) \in H(0)$ &$(\soo(7),G_2,0)$& \\
\hline
\multicolumn{6}{|c|}{$D_{n+1}^{(2)}$} \\
\hline
{\tiny D12)}&$[2](\aa_0+\dots+\aa_d)\in H(d+1)$ &$(\soo(2d+3),\soo(2d+2),0)$& $[2](\aa_{d+1}+\dots+\aa_n)\in H(d)$      &$(\soo(2n-2d),\soo(2n-2d-1),0)$&  $0\le d < n$  \\
\hline
{\tiny D13)}&$\aa_1+\dots+\aa_{n-1}\in H(0,n)$ &$(\sll(n-1),\gll(n-2),0)$& $\aa_0+\aa_n\in H(1,n-1)$&$(\sll(2)\times \sll(2),\Delta\sll(2),0)$& \\
\hline
\addlinespace[0.3cm]
\hline
\multicolumn{6}{|c|}{$E_{6}^{(2)}$} \\
\hline

E1)&$2\aa_1+3\aa_2+2\aa_3+\aa_4\in H(0)$ &$(F_4,\soo(9),0)$& $\aa_0\in H(1)$      &$(\sll(2),0,0)$&   \\
\hline
E2)&$\aa_2+\aa_3+\aa_4\in H(1)$ &$(\soo(7),\soo(6),0)$&  $\aa_0+2\aa_1+2\aa_2+\aa_3\in H(4)$ &$(\sp(8),(\sll(2)+\sp(6)),0)$& \\
\hline
E3)&{$\aa_0+\aa_1+\aa_2\in H(3)$} &$(\sll(4),\gll(3),0)$&$\boldsymbol{-\w\in I_3(0,2)}$ &$(\gll(2),\gll(2),C^2)$& \\
\hline
\pagebreak
\hline
\multicolumn{6}{|c|}{$D_{4}^{(3)}$} \\
\hline
{\tiny D14)}&$[2](2\aa_1+\aa_2) \in H(0)$ &$(G_2,\sll(3),0)$& $[2]\aa_0\in H(1)$      &$(\sll(2),0,0)$&   \\
\hline
{\tiny D15)}&$[2]\frac 12(\aa_0+\aa_2) \in H(1)$ &$(\sll(2)\times \sll(2),\Delta\sll(2),0)$& $[2]\aa_1\in H(0,2)$ &$(\sll(2),0,0)$& \\
\hline
{\tiny D16)}&{$\aa_0+\aa_1\in H(2)$} &$(\sll(3),\gll(2),0)$& $\boldsymbol{-\w\in I_2(0,1)}$ &$(\sll(2),\sll(2),\C^2)$& \\
\hline
\end{longtable}
}
\end{center}
\end{lis}

We continue with a list of the corresponding diagrams:

\begin{longtable}{c|c|c|c|c}
$\begin{picture}(2400,1800)(-300,-900)
\put(0,0){\usebox{\leftrightbiedge}}
\put(0,0){\usebox{\aprime}}
\put(1800,0){\usebox{\aprime}}
\end{picture}$&$\begin{picture}(4500,3600)(-2350,-900)
        \put(-2500,600){\tiny{[1/2]}}
         \multiput(0,0)(0,1800){2}{\usebox{\edge}}
         \multiput(0,1800)(1800,0){2}{\usebox{\vedge}}
         \put(0,0){\usebox{\wcircle}}
         \put(1800,0){\usebox{\wcircle}}
         \put(0,1800){\usebox{\wcircle}}
         \put(1800,1800){\usebox{\wcircle}}
         \put(200,200){\line(1,1){1400}}
         \put(200,1600){\line(1,-1){1400}}
        \end{picture}$
       &$\begin{picture}(4500,1800)(-2350,-900)
        \put(-2500,600){\tiny{[1/2]}}
         \multiput(0,0)(0,1800){2}{\usebox{\edge}}
         \multiput(0,1800)(1800,0){2}{\usebox{\vedge}}
         \put(0,0){\circle*{600}}
         \put(1800,1800){\circle*{600}}
        \end{picture}$&$\begin{picture}(10800,3600)(-600,-900)
        \thicklines
         \multiput(0,0)(5400,0){2}{\usebox{\shortam}}
         \put(3600,0){\usebox{\edge}}
           \multiput(-600,0)(9600,0){2}{\line(1,0){600}}
\multiput(-600,0)(10200,0){2}{\line(0,-1){1200}}
\put (-650,-1200){\line(1,0){10250}}
        \end{picture}$ & 
        $
        \begin{picture}(10800,1800)(-600,-900)
        \thicklines
         \multiput(0,0)(5400,0){2}{\usebox{\dynkinathree}}
         \put(3600,0){\usebox{\shortsusp}}
           \multiput(-600,0)(9600,0){2}{\line(1,0){600}}
\multiput(-600,0)(10200,0){2}{\line(0,-1){1200}}
\put (-650,-1200){\line(1,0){10250}}
\put(-350,-400){$\btr$}
\put(8250, -400){$\btl$}

        \end{picture}
       $\\
A1)&A2)&A3)&A4)&A5\\
\hline
 $\begin{picture}(4500,3600)(0,-1800)
 \thicklines
\put(0,0){\usebox\dynkinbthree}
\put(1800,0){\usebox\vedge}
\put(3600,0){\circle*{600}}
\put(-350,-350){$\btr$}
\put(3300,600){\tiny $\nicefrac12$}
 \end{picture}$
 &
$\begin{picture}(4500,3600)(-900,-900)
\thicklines
\put(0,0){\usebox\dynkinbthree}
\put(1800,0){\usebox\vedge}
\put(0,0){\circle{600}}
\put(3700,0){\circle{600}}
\multiput(0,300)(3700,0){2}{\line(0,1){300}}
\put(0,600){\line(1,0){3700}}
\put(1800,1200){\tiny $\nicefrac12$}
\put(1300,-250){$\btd$}
 \end{picture}$&&$\begin{picture}(2400,2700)(1500,-1800)
        \put(-2100,0){\tiny{[2]}}
         \multiput(0,0)(1800,0){2}{\usebox{\vertex}}
         \put(1800,-1800){\usebox{\vertex}}
         \put(0,0){\line(1,0){1800}}
         \put(1800,0){\line(0,-1){1800}}
         \put(1800,0){\usebox{\susp}}
         \put(5400,0){\usebox{\rightbiedge}}
         \put(0,0){\circle*{600}}
         \put(1800,-1800){\circle*{600}}
        \end{picture}$&
        $\begin{picture}(2400,2700)(3300,-900)
        \put(0,1200){[\nicefrac12]}
         \multiput(0,0)(1800,0){2}{\usebox{\vertex}}
         \multiput(1800,-1800)(1800,1800){2}{\usebox{\vertex}}
         \multiput(0,0)(1800,0){2}{\line(1,0){1800}}
         \put(1800,0){\line(0,-1){1800}}
         \put(3600,0){\usebox{\shortsusp}}
         \put(5400,0){\usebox{\edge}}
         \put(7200,0){\usebox{\shortsusp}}
         \put(9000,0){\usebox{\rightbiedge}}
         \put(5400,0){\circle*{600}}
         \put(7200,0){\circle*{600}}
         \put(7200,600){\tiny 2}
         \end{picture}$\\
 B1) & B2) && B3) & B4)\\
 \hline
$
        \begin{picture}(3600,2400)(-300,-900)
        \put(200,500){\tiny{[\nicefrac12]}}
       \put(0,0){\usebox{\rightbiedge}}
       \put(1800,0){\usebox{\leftbiedge}}
       \multiput(0,0)(3600,0){2}{\usebox{\wcircle}}
       \put(1800,-600){\usebox{\wcircle}}
       \multiput(0,-250)(3600,0){2}{\line(0,-1){950}}
       \put(0,-1200){\line(1,0){3600}}
         \end{picture}
       $&
$
        \begin{picture}(5400,1800)(-1500,-900)
        \put(-1800,0){\tiny{[2]}}
       \put(0,0){\usebox{\rightbiedge}}
       \put(1800,0){\usebox{\leftbiedge}}
       \multiput(0,0)(3600,0){2}{\circle*{600}}
         \end{picture}$&
        
         $
        \begin{picture}(7500,2400)(-300,-900)
        \put(3000, -900){\tiny 1/2}
        \put(0,0){\usebox{\rightbiedge}}
       \put(5400,0){\usebox{\leftbiedge}}
       \put(1800,0){\usebox{\shortam}}
       \multiput(0,0)(7200,0){2}{\usebox{\wcircle}}
       \multiput(0,-250)(7200,0){2}{\line(0,-1){950}}
       \put(0,-1200){\line(1,0){7200}}
       \end{picture}
       $&
       
$
        \begin{picture}(9000,1800)(900,-900)
        \put(0,0){\usebox{\rightbiedge}}
       \put(9000,0){\usebox{\leftbiedge}}
       \put(1800,0){\usebox{\shortsusp}}
       \put(3600,0){\usebox{\dynkinathree}}
       \put(7200,0){\usebox{\shortsusp}}
       \multiput(3600,0)(3600,0){2}{\circle*{600}}
       \end{picture}
       $&$
        \begin{picture}(7200,2400)(-300,-900)
       \put(0,0){\usebox{\rightbiedge}}
       \put(5400,0){\usebox{\leftbiedge}}
       \put(1800,0){\usebox{\susp}}
       \put(1500,-350){$\btr$}
       \put(5199,-350){$\btl$}
       \end{picture}
       $ 
       \\
       C1) & C2) & C3) & C4) & C5\\
       \hline
$\begin{picture}(2400,5000)(1500,-1500)
\put(1800,2100){\tiny[1/2]}        
\multiput(900,0)(1800,0){2}{\usebox{\edge}}
\put(2700,0){\usebox{\vedge}}
\put(2700,1800){\usebox{\vedge}}
\put(2700,1800){\usebox{\vertex}}
\put(900,0){\circle*{600}}
\put(2700,-1800){\circle*{600}}
\end{picture}$

&$\begin{picture}(5400,5000)(-1500,-1500)
\put(-1800,600){\tiny[1/2]}        
\multiput(0,0)(1800,0){2}{\usebox{\edge}}
\put(1800,0){\usebox{\vedge}}
\put(1800,1800){\usebox{\vedge}}
\put(1800,1800){\usebox{\vertex}}
\multiput(1800,1800)(0,-3600){2}{\circle{600}}
\put(1800,0){\circle*{600}}
\multiput(2050,1800)(0,-3600){2}{\line(1,0){2150}}
\put(4200, 1800){\line(0,-1){3600}}
\end{picture}$&$ \begin{picture}(2400,5000)(-0000,-300)
 \put(-2700,-200){\tiny[1/2]}
\put(0,0){\usebox{\vertex}}
\multiput(-1200,1200)(0,-2400){2}{\usebox{\vertex}}
\put(-1200,-1200){\line(1,1){1200}}
\put(-1200,1200){\line(1,-1){1200}}
\put(000,0){\usebox{\susp}}
\put(3600,0){\usebox{\bifurc}}
\multiput(-1200,1200)(0,-2400){2}{\circle*{600}}
\end{picture}$ & $\begin{picture}(9000,1800)(00,-900)
\put(-900,2400){\tiny[1/2]}
        \put(0,0){\usebox{\vertex}}
\multiput(-1200,1200)(0,-2400){2}{\usebox{\vertex}}
\put(-1200,-1200){\line(1,1){1200}}
\put(-1200,1200){\line(1,-1){1200}}
\put(0,0){\usebox{\shortsusp}}
\put(1800,0){\usebox{\dynkinafour}}
\put(7200,0){\usebox{\shortsusp}}
\put(9000,0){\usebox{\bifurc}}
\multiput(3600,0)(1800,0){2}{\circle*{600}}
\end{picture}$&$\begin{picture}(9000,5000)(-2100,-900)
\put(0,0){\usebox{\vertex}}
\multiput(-1200,1200)(0,-2400){2}{\usebox{\vertex}}
\put(-1200,-1200){\line(1,1){1200}}
\put(-1200,1200){\line(1,-1){1200}}
\put(000,0){\usebox{\susp}}
\put(3600,0){\usebox{\bifurc}}
\multiput(-1200,1200)(0,-2400){2}{\circle{600}}
\multiput(4800,1200)(0,-2400){2}{\circle{600}}
\multiput(-1200,-1200)(25,0){20}{\circle*{70}}
\multiput(-700,-1200)(0,25){12}{\circle*{70}}
\multiput(-700,-900)(25,0){12}{\circle*{70}}
\multiput(-400,-900)(0,25){12}{\circle*{70}}
\multiput(-400,-600)(25,0){12}{\circle*{70}}
\multiput(-100,-600)(0,25){12}{\circle*{70}}
\multiput(-100,-300)(25,0){12}{\circle*{70}}
\multiput(200,-300)(25,25){8}{\circle*{70}}
\multiput(400,-100)(25,-25){8}{\circle*{70}}
\multiput(600,-300)(25,25){8}{\circle*{70}}
\multiput(800,-100)(25,-25){8}{\circle*{70}}
\multiput(1000,-300)(25,25){8}{\circle*{70}}
\multiput(1200,-100)(25,-25){8}{\circle*{70}}
\multiput(1400,-300)(25,25){8}{\circle*{70}}
\multiput(1600,-100)(25,-25){8}{\circle*{70}}
\multiput(1800,-300)(25,25){8}{\circle*{70}}
\multiput(2000,-100)(25,-25){8}{\circle*{70}}
\multiput(2200,-300)(25,25){8}{\circle*{70}}
\multiput(2400,-100)(25,-25){8}{\circle*{70}}
\multiput(2600,-300)(25,25){8}{\circle*{70}}
\multiput(2800,-100)(25,-25){8}{\circle*{70}}
\multiput(3000,-300)(25,25){8}{\circle*{70}}
\multiput(3200,-100)(25,-25){8}{\circle*{70}}
\multiput(3400,-300)(25,0){12}{\circle*{70}}
\multiput(3700,-300)(0,-25){12}{\circle*{70}}
\multiput(3700,-600)(25,0){12}{\circle*{70}}
\multiput(4000,-600)(0,-25){12}{\circle*{70}}
\multiput(4000,-900)(25,0){12}{\circle*{70}}
\multiput(4300,-900)(0,-25){12}{\circle*{70}}
\multiput(4300,-1200)(25,0){20}{\circle*{70}}
\multiput(-1200,+1200)(25,0){20}{\circle*{70}}
\multiput(-700,+1200)(0,-25){12}{\circle*{70}}
\multiput(-700,+900)(25,0){12}{\circle*{70}}
\multiput(-400,900)(0,-25){12}{\circle*{70}}
\multiput(-400,600)(25,0){12}{\circle*{70}}
\multiput(-100,600)(0,-25){12}{\circle*{70}}
\multiput(-100,300)(25,0){12}{\circle*{70}}
\multiput(200,300)(25,-25){8}{\circle*{70}}
\multiput(400,100)(25,25){8}{\circle*{70}}
\multiput(600,300)(25,-25){8}{\circle*{70}}
\multiput(800,100)(25,25){8}{\circle*{70}}
\multiput(1000,300)(25,-25){8}{\circle*{70}}
\multiput(1200,100)(25,25){8}{\circle*{70}}
\multiput(1400,300)(25,-25){8}{\circle*{70}}
\multiput(1600,100)(25,25){8}{\circle*{70}}
\multiput(1800,300)(25,-25){8}{\circle*{70}}
\multiput(2000,100)(25,25){8}{\circle*{70}}
\multiput(2200,300)(25,-25){8}{\circle*{70}}
\multiput(2400,100)(25,25){8}{\circle*{70}}
\multiput(2600,300)(25,-25){8}{\circle*{70}}
\multiput(2800,100)(25,25){8}{\circle*{70}}
\multiput(3000,300)(25,-25){8}{\circle*{70}}
\multiput(3200,100)(25,25){8}{\circle*{70}}
\multiput(3400,300)(25,0){12}{\circle*{70}}
\multiput(3700,300)(0,25){12}{\circle*{70}}
\multiput(3700,600)(25,0){12}{\circle*{70}}
\multiput(4000,600)(0,25){12}{\circle*{70}}
\multiput(4000,900)(25,0){12}{\circle*{70}}
\multiput(4300,900)(0,25){12}{\circle*{70}}
\multiput(4300,1200)(25,0){20}{\circle*{70}}
\end{picture}$ 
\\
D1) & D2) &   D4)& D5) & D3), D6)\\ 
\hline
&&$\begin{picture}(2400,1800)(2400,-900)
\put(0,0){\usebox{\dynkinatwo}}
\put(1800,0){\usebox{\dynkinffour}}
\put(0,0){\circle*{600}}
\put(7200,0){\circle*{600}}
\end{picture}
$
&
$
        \begin{picture}(2400,1800)(2400,-200)
\put(1200,600){\tiny[1/2]}        
\put(0,0){\usebox{\dynkinatwo}}
\put(1800,0){\usebox{\dynkinffour}}
\put(1800,0){\circle*{600}}
\put(5400,0){\circle*{600}}
\end{picture}
$ &$\begin{picture}(2400,1800)(2400,-900)
\put(0,0){\usebox{\dynkinatwo}}
\put(1800,0){\usebox{\dynkinffour}}
\put(5400,0){\usebox\atwo}
\put(3300,-350){$\btl$}
\end{picture}
$\\
&&F1) & F2)&F3)\\
\hline
$ \begin{picture}(3600,1800)(0,-900)
 \put(0,0){\usebox\dynkinathree}
 \multiput(1800,200)(0,-400){2}{\line(1,0){1800}}
 \put(3300,0){\line(-1,1){500}}
  \put(3300,0){\line(-1,-1){500}}
  \put(3600,0){\circle*{600}}
  \put(-300,-350){$\btr$}
 \end{picture}$
 &
$  \begin{picture}(3600,1800)(0,-500)
 \put(0,0){\usebox\dynkinathree}
 \multiput(1800,200)(0,-400){2}{\line(1,0){1800}}
 \put(3300,0){\line(-1,1){500}}
  \put(3300,0){\line(-1,-1){500}}
  \put(3600,0){\usebox\aone}
  \put(1300,-350){$\btl$}
 \end{picture}$
 &
$\begin{picture}(3600,1800)(0,-500)
 \put(0,0){\usebox\dynkinathree}
 \multiput(1800,200)(0,-400){2}{\line(1,0){1800}}
 \put(3300,0){\line(-1,1){500}}
  \put(3300,0){\line(-1,-1){500}}
  \multiput(0,0)(3600,0){2}{\circle{600}}
  \multiput(0,300)(3600,0){2}{\line(0,1){600}}
  \put(0,900){\line(1,0){3600}}
  \put(1400,-200){$\bt$}
\end{picture}$&&\\
G1) & G2) & G3)&&\\
\hline
$\begin{picture}(2400,2400)(-300,-900)
\multiput(0,0)(1800,0){2}{\circle*{300}}\thicklines
\multiput(0,-60)(0,120){2}{\line(1,0){1800}}
\multiput(0,-180)(0,360){2}{\line(1,0){1800}}
\multiput(150,0)(25,25){20}{\circle*{50}}
\multiput(150,0)(25,-25){20}{\circle*{50}}
\put(0,0){\usebox{\aprime}}
\put(1800,0){\usebox{\aone}}
\end{picture}$ & 
$\begin{picture}(2400,2400)(-300,-900)
\multiput(0,0)(1800,0){2}{\circle*{300}}\thicklines
\multiput(0,-60)(0,120){2}{\line(1,0){1800}}
\multiput(0,-180)(0,360){2}{\line(1,0){1800}}
\multiput(150,0)(25,25){20}{\circle*{50}}
\multiput(150,0)(25,-25){20}{\circle*{50}}
\put(1500,-350){$\btl$}
\put(0,0){\usebox{\aone}}
\end{picture}$&$\begin{picture}(7200,2400)(-300,-900)
        \multiput(0,0)(5400,0){2}{\usebox{\leftbiedge}}
        \put(1800,0){\usebox{\susp}}
        \put(0,0){\usebox{\vertex}}
        \put(5400,0){\circle*{600}}
        \put(5400,600){\tiny 2}
        \put(7200,0){\usebox{\aone}}
\end{picture}$
& 
$\begin{picture}(7200,1800)(1500,-900)
        \multiput(0,0)(9000,0){2}{\usebox{\leftbiedge}}
        \put(1800,0){\usebox{\susp}}
        \put(3600,0){\usebox\dynkinatwo}
        \put(5400,0){\usebox\susp}
        \put(0,0){\usebox{\vertex}}
        \put(3600,0){\circle*{600}}
        \put(5300,-350){$\btr$}
\end{picture}$&$\begin{picture}(2400,2400)(-300,-900)
        \put(-1800,0){\usebox{\edge}}
        \put(0,0){\usebox{\vedge}}
        \put(0,0){\usebox{\shortsusp}}
        \put(1800,0){\usebox{\leftbiedge}}
        \put(-1800,0){\usebox{\aone}}
        \put(0,0){\circle*{600}}
        \end{picture}
$\\
A6) & A7)& A8) & A9)&A10)\\
\hline
&$
        \begin{picture}(3600,3600)(-900,-900)
        \put(-1800,0){\usebox{\edge}}
        \put(0,0){\usebox{\vedge}}
        \put(0,0){\usebox{\shortsusp}}
        \put(1800,0){\usebox{\leftbiedge}}
        \put(1800,0){\circle*{600}}
        \put(3600,0){\usebox{\aone}}
        \end{picture}$&
        
       $\begin{picture}(3600,3600)(-300,-900)
        \put(-1800,0){\usebox{\edge}}
        \put(0,0){\usebox{\vedge}}
        \put(0,0){\usebox{\edge}}
        \put(1800,0){\usebox{\shortsusp}}
        \put(3600,0){\usebox{\leftbiedge}}
        \put(-1800,0){\circle{600}}
        \put(0,-1800){\circle{600}}
        \put(1800,0){\circle*{600}}
        \multiput(-1800,0)(25,-25){13}{\circle*{70}}
        \multiput(-1475,-325)(25,25){7}{\circle*{70}}
        \multiput(-1300,-150)(25,-25){7}{\circle*{70}}
        \multiput(-1125,-325)(25,25){7}{\circle*{70}}
        \multiput(-950,-150)(25,-25){7}{\circle*{70}}
        \multiput(-775,-325)(25,25){7}{\circle*{70}}
        \multiput(-600,-150)(25,-25){7}{\circle*{70}}
        \multiput(-425,-325)(25,25){7}{\circle*{70}}
        \multiput(-250,-150)(25,-25){7}{\circle*{70}}
        \multiput(-75,-325)(-25,-25){7}{\circle*{70}}
        \multiput(-250,-500)(25,-25){7}{\circle*{70}}
        \multiput(-75,-675)(-25,-25){7}{\circle*{70}}
        \multiput(-250,-850)(25,-25){7}{\circle*{70}}
        \multiput(-75,-1025)(-25,-25){7}{\circle*{70}}
        \multiput(-250,-1200)(25,-25){7}{\circle*{70}}
        \multiput(-75,-1375)(-25,-25){7}{\circle*{70}}
        \multiput(-250,-1550)(25,-25){12}{\circle*{70}}
\end{picture}$& $\begin{picture}(7200,3000)(-300,-1900)
        \put(-1800,0){\usebox{\edge}}
        \put(0,0){\usebox{\vedge}}
        \put(0,0){\usebox{\edge}}
        \put(1800,0){\usebox{\susp}}
        \put(5400,0){\usebox{\leftbiedge}}
        \put(-2100,-350){$\btr$}
        \put(-500,-2100){$\bt$}
        \end{picture}$
        &
$\begin{picture}(7200,3000)(-300,-1900)
        \put(-1800,0){\usebox{\edge}}
        \put(0,0){\usebox{\vedge}}
        \put(0,0){\usebox{\edge}}
        \put(1800,0){\usebox{\shortsusp}}
        \put(3600,0){\usebox{\dynkinatwo}}
        \put(5400,0){\usebox{\shortsusp}}
        \put(7200,0){\usebox{\leftbiedge}}
        \put(3600,0){\circle*{600}}
        \put(5100,-350){$\btr$}
        \end{picture}$\\
         & A11) & A12)&A13)&A14)\\
\hline
$        \begin{picture}(3600,2400)(-300,00)
        \put(0,0){\usebox{\vertex}}
        \put(0,0){\usebox{\leftbiedge}}
        \put(1800,0){\usebox{\rightbiedge}}
		\multiput(0,0)(3600,0){2}{\circle{600}}
		\multiput(0,300)(3600,0){2}{\line(0,1){300}}
		\put(0,600){\line(1,0){3600}}
		\put(1500,1200){\tiny $\nicefrac12$}
		\put(1500,-350){$\btl$}
        \end{picture}$&
        $        \begin{picture}(3600,2400)(0,-900)
        \put(0,0){\usebox{\vertex}}
        \put(0,0){\usebox{\leftbiedge}}
        \put(1800,0){\usebox{\rightbiedge}}
		\put(-200,-350){$\btr$}
		\put(3400,-350){$\btl$}
        \end{picture}$&$        \begin{picture}(3600,2400)(-150,-900)
        \put(-2400,0){\tiny[1/2]}
        \put(0,0){\usebox{\vertex}}
        \put(0,0){\usebox{\leftbiedge}}
        \put(3600,0){\usebox{\rightbiedge}}
        \put(1800,0){\usebox{\dynkinatwo}}
        \multiput(0,0)(3600,0){2}{\usebox{\wcircle}}
        \multiput(1800,0)(3600,0){2}{\usebox{\wcircle}}
        \multiput(0,-250)(3600,0){2}{\line(0,-1){950}}
        \multiput(1800,250)(3600,0){2}{\line(0,1){950}}
        \put(0,-1200){\line(1,0){3600}}
        \put(1800,1200){\line(1,0){3600}}
        \end{picture}$
        &
        $\begin{picture}(2400,1800)(1500,-900)
         \put(-2400,0){\tiny[1/2]}
        \put(0,0){\usebox{\vertex}}
        \put(0,0){\usebox{\leftbiedge}}
        \put(3600,0){\usebox{\rightbiedge}}
        \put(1800,0){\usebox{\dynkinatwo}}
        \multiput(0,0)(5400,0){2}{\circle*{600}}
        \end{picture}$\\
        D8) & D9) &D10)& D11& \\
\hline
&&$
        \begin{picture}(9000,1800)(-300,-00)
        \put(0,900){\tiny [2]}
        \put(0,0){\usebox{\vertex}}
        \put(0,0){\usebox{\leftbiedge}}
        \put(7200,0){\usebox{\rightbiedge}}
        \multiput(1800,0)(3600,0){2}{\usebox{\shortsusp}}
        \put(3600,0){\usebox{\dynkinatwo}}
        \put(3600,0){\circle*{600}}
        \put(5400,0){\circle*{600}}
        \end{picture}
$
&
$
        \begin{picture}(9000,1800)(-300,00)
        \put(0,0){\usebox{\vertex}}
        \put(0,0){\usebox{\leftbiedge}}
        \put(5400,0){\usebox{\rightbiedge}}
        \put(1800,0){\usebox{\shortam}}
        \multiput(0,0)(7200,0){2}{\usebox{\wcircle}}
        \multiput(0,250)(7200,0){2}{\line(0,1){950}}
        \put(0,1200){\line(1,0){7200}}
  \end{picture}
$&\\
&&D12) & D13)        & \\
\hline
&&$
        \begin{picture}(9000,2400)(-300,-900)
        \put(0,0){\usebox{\dynkinathree}}
        \put(3600,0){\usebox{\leftbiedge}}
        \put(5400,0){\usebox{\dynkinatwo}}
        \put(0,0){\usebox{\aone}}
        \put(1800,0){\circle*{600}}
  \end{picture}
$
&
$
        \begin{picture}(9000,1800)(-300,-900)
        \put(0,0){\usebox{\dynkinathree}}
        \put(3600,0){\usebox{\leftbiedge}}
        \put(5400,0){\usebox{\dynkinatwo}}
        \put(7200,0){\circle*{600}}
        \put(1800,0){\circle*{600}}
  \end{picture}
$
&
$
        \begin{picture}(9000,1800)(-300,-900)
        \put(0,0){\usebox{\athree}}
        \put(3600,0){\usebox{\leftbiedge}}
        \put(5400,0){\usebox{\dynkinatwo}}
        \put(5100,-350){$\btr$}
  \end{picture}
$\\
&&E1) & E2) & E3)\\
\hline
&&$
        \begin{picture}(2400,2400)(-300,-900)
        \put(-3600,0){\tiny{[1/2]}}
        \put(0,0){\usebox{\dynkinatwo}}
        \put(1800,0){\usebox{\dynkingtwo}}
        \put(0,0){\usebox{\aprime}}
        \put(1800,0){\circle*{600}}
        \put(1800,600){\tiny 2}
  \end{picture}
$
&
$
        \begin{picture}(2400,1800)(-300,-900)
        \put(-3600,0){\tiny{[1/2]}}
        \put(0,0){\usebox{\dynkinatwo}}
        \put(1800,0){\usebox{\dynkingtwo}}
        \put(1800,0){\usebox{\aprime}}
        \multiput(0,0)(3600,0){2}{\usebox{\wcircle}}
        \multiput(0,-250)(3600,0){2}{\line(0,-1){950}}
        \put(0,-1200){\line(1,0){3600}}
  \end{picture}
$
&
$
        \begin{picture}(2400,2400)(-300,-900)
        \put(-3600,0){\tiny{[1/2]}}
        \put(0,0){\usebox{\atwo}}
        \put(1800,0){\usebox{\dynkingtwo}}
        \put(3300,-250){$\btr$}

  \end{picture}
$\\
&&D14) & D15) & D16)\\
\hline

\end{longtable}

\begin{rem}\leavevmode\\
\begin{enumerate}
\item Knop showed in \cite{Kno14} that the case A2 (without the factor 1/2) corresponds to the disymmetric manifold $\nicefrac{\SU(4)}{SO(4)}\times \nicefrac{\SU(4)}{\SP(4)}$.
\item The three possible cases $\w(A5),2\w(A4),4\w(A1)$ for $A_1^{(1)}$ correspond, in that order, to the manifolds $S^4$ (the so-called "spinning 4-sphere", \cite{AMW02}), $S^2\times S^2$ and $\PP^2(\C)$ by \cite{Kno14}, pp. 36.
\item Knop \cite{Kno14} determined the Manifold that corresponds to case C5). More precisely, this spherical pair comes from the $\SP(2n)$-manifold structure on the quaternionic Grassmanian: $M=\operatorname{Gr}_k(\mathbb{H}^{n+1})$. This is a generalization of a result from \cite{Esh09}.
\end{enumerate}
\end{rem}

\part{Hamiltonian Manifolds of Rank one}

Let us recall that every q-Hamiltonian manifold localy looks like a Hamiltonian manifold, cf. \cite{Kno14}. Conversely, every Hamiltonian manifolds carries a quasi-Hamiltonian structure. We shall use this to apply our theory to classify momentum polytopes of compact Hamiltonian manifolds of rank one.\\

Hence, speaking in the quasi-Hamiltonian setting, we now want to look at momentum polytopes of rank one that do not hit the wall of the fundamental alcove corresponding to $\a_0$, but every other wall. That means we can forget about the existence of this one wall and find ourselves in the Weyl chamber of a classical root system that is in bijection with the orbit space $\mathfrak{k^*}/K$ of a classical Hamiltonian manifold, more details in \cite{Kno11}. We again have the classification theorem that this manifolds are characterized by spherical pairs.\\

We classify all compact convex Hamiltonian manifolds such that the corresponding spherical pair $(\P,\L_S)$ has the following properties:

\begin{itemize}
\item The moment polytope $\P=[X_1X_2]$ is a line segment touching every wall of the dominant chamber, meaning $\a(X_1)=0$ or $\a(X_2)=0$ for everey simple root $\a\in\{\a_1, \dots, \a_n\}$, that does not touch $\mathbf 0$.
\item The lattice $\L_S$ has rank one and its generator $\w$ generates $\P$ as a line segment.
\end{itemize}

Most of our theory now works similar. We have the following structure theorems:

\begin{thm} 
There are no bi-homogeneous polytopes of rank one for compact Hamiltonian $K$-manifolds.
\end{thm}

\begin{proof}
Suppose there was a bi-homogeneous polytope. Hence the weight $\w_1$ in $X_1$ and the weight $\w_2$ in $X_2$ are $\Z$-linear combinations of simple roots. As we suppose our polytope to hit every wall of the alcove, $\w_2=-\w_1$ could only be possible if there was a linear dependence between the simple roots (Which is given via the root $\d$ in the affine case!). But as the simple roots of a classical root system are well known to be linearly independent, this is not possible.
\end{proof}

We can also find an analogon for \cref{inthm}:
\begin{thm} 
 Let $G$ be simple of rank $n$ and $\w$ the weight of an inhomogeneous smooth affine spherical variety of rank one that is a local model in $X_1$ with $\la \w, \a_k^\vee \ra =1$. Then, for $X_2=X_1+c\w$, we have
 \[
   S(X_2)= S \setminus \{\a_k\}
 \]

\end{thm}

\begin{proof}
 This is just the proof of \cref{inthm} with some slight adjustments: We again assume $S(X_1) =\{\a_1, \dots, \a_k\}$ after renumbering. This time, our $\w$ has $n-k$ free coordinates, and $X_1$ also has $n-k$ free parameters. We argue as above that $\a_1, \dots, \a_{k-1}\in  S(X_2), \a_k \notin  S(X_2)$. It is also possible for the polytope to touch every wall of the alcove because for that, it is necessary to solve a homogeneous system of $n-k$ equations with $2(n-k)+1$ free parameters. As the equations are determined by simple roots (which are linearly independent), the system is always solvable.
\end{proof}

We deduce the same way as for genuine q-Hamiltonians:

\begin{cor} \leavevmode\\
 \begin{enumerate}
  \item Bi-inhomogeneous polytopes are only possible if $|S(X_1)|=| S(X_2)|=n-1$ and, if $S(X_1)=S \setminus \a_k , \la \w, \a_j^\vee \ra =1$ we must have $ S(X_2)=S \setminus \a_j$ and $\la -\w, \a_k^\vee \ra =1$.
  \item Let $\w$ be a homogeneous polytope for $X_1$. Mixed polytopes are only possible if $|S(X_1)|=n-1$.
 \end{enumerate}

\end{cor}

\begin{proof}\leavevmode\\
 \begin{enumerate}
  \item Reason the same way as in the proof of \cref{incor1}
  \item Same reasons as for \cref{mixcor2}
 \end{enumerate}
\end{proof}

It follows:

\begin{cor} 
Let us consider a compact Hamiltonian $K$-manifold  of rank one for $K$ simple such that the corresponding moment polytope touches every wall of the chamber, but not in the origin. Then we can find its spherical pair by examining all $X_1$ with $|S(X_1)|=n-1$.
\end{cor}

We go on and explicitly determine these polytopes, using the same short notation we used for the q-Hamiltonians. If something does not work, the reason is always \cref{inh1}. We shall again give the case-by-case-study for $A_n$ and refer to \cite{Pau} for details on the other cases (that work very similar).

For $A_n$, we have: 

\begin{center}
\tiny{
\begin{longtable}{|p{3cm}|p{3cm}|p{3cm}|p{4cm}|}
\hline
$S(X_1)$ & choice of $\w$ & $ S(X_2)$ &  $-w$ inhom.?\\\endhead
\hline
$\sm{1}$ & $\aa_{1,n}$ & $\ssm{1}{n}$& $-\w\in I_1$  \\
\hline
$n=2$ & 2$\a_2$ &$\sm{2}$ & no: $\Sp{1}=2$  \\
\hline
$n=4$ & $\aa_2+2\aa_3+\aa_4$ & $\sm{3}$ & $-w\in I_4$ \\
\hline
$n=4$ & $\frac12(\aa_2+2\aa_3+\aa_4)$ & $\sm{3}$ & no: $\Sp{4}=\frac12$,  \\
\hline
$\sm{d}, 0<d<n$ & $\a_{1,d}$ & $\sm{1,d}$ & $-\w\in I_{d-1}$ \\
\hline
$2<d<n-1$& $I_1$ & $\sm{1}$ & no: $\Sp{d}>0$ \\
\hline
& $I_{d-1}$ & $\sm{d-1}$ & $-\w\in I_d$ \\
\hline
$d=2$ & $I_1$ & $\sm{1}$ & $-\w\in I_2$ \\
\hline
$d=4$ & $[\frac12]\aa_1+2\aa_2+\aa_3$ & $\sm{2}$ &no: $\Sp{4}>0$\\
\hline
$d=2$ & $2\a_1$ & $\sm{1}$ & no: $\Sp{2}=2$ \\
\hline
$n=3, d=2$ & $\aa_1+\aa_3$ & $\a_2$ & no: $\Sp{2}=2$ \\
\hline
& $\frac12(\aa_1+\aa_3)$ & $\a-2$ & $-\w\in I_2$\\
\hline
\end{longtable}
}
\end{center}

\section{List of Hamiltonian manifolds of rank one}


\begin{thm}[Momentum Polytopes of Rank one] 
 A compact multiplicity free Hamiltonian $K$-manifold of rank one for $K$ simple with moment polytope $\P=[X_1X_2]$ such that $\a(X_1)=0$ or $\a(X_2)=0$ for every simple root $\a_1, \dots, \a_n$ corresponds to a spherical pair $(\P,\L_S)$ with the following properties:
 \begin{itemize}
\item $\P$ is a line segment $[X_1X_2]:=(X_1+\R\w)\cap \bar\A$ for some $\w$ found in the list below and $X_1$ such that $S(X_1)$ is the local root system indicated.
\item The lattice $\L_S$ is generated by this $\w$ 

\end{itemize}
\end{thm}

\begin{lis}\label{ham1} List of Hamiltonian manifolds of rank one
\tiny{
\begin{longtable}{|p{0.7cm}|p{2cm}|p{3.2cm}|p{2cm}|p{3.2cm}||p{1.5cm}|}
\hline
&$\w$ &local model & $-\w$ &local model & remarks\\\endhead
\hline
\multicolumn{6}{|c|}{$A_{3}$} \\
\hline
A1)&$\frac 12 (\a_1+\a_3)\in H(2)$ &$(\sll(2)\times \sll(2),\Delta\sll(2),0)$& $\boldsymbol{-w\in I_2(1,3)}$     &$(\tt^1+\sll(2)+\tt^1,\tt^1+\sll(2)+\tt^1,\C^2_{\nicefrac12,-\nicefrac12})$ &   \\
\hline
\multicolumn{6}{|c|}{$A_{4}$} \\
\hline
A2)&$\a_1+2\a_2+\a_3\in H(4)$ &$(\sll(4),\sp(4),0)$& $\boldsymbol{-w\in I_4(2)}$    &$(\tt^1+\sll(3),\tt^1+\sll(3),\C^3_2)$  &   \\
\hline
\multicolumn{6}{|c|}{$A_{n}$} \\
\hline
A3)&$\a_1+\dots+\a_k\in H(k+1)$ &$(\sll(k),\gll(k-1),0)$ & $\boldsymbol{-w\in}\linebreak \boldsymbol{I_{k+1}(1,k)}$   &$(\tt^1+\sll(n-k)+\tt^1,\tt^1+\sll(n-k)+\tt^1,\C^{n-k}_{1,1})$  &   \\
\hline
A4)&$ \boldsymbol{\w\in I_k(k+1)}$ &$(\tt^1+\sll(k),\tt^1+\sll(k),\C^k_{-\nicefrac{k}{n+1}})$ & $\boldsymbol{-\w\in I_{k+1}(k)}$ & $(\tt^1+\sll(n-k),\tt^1+\sll(n-k),\C^{n-k}_{\nicefrac{n-k}{n+1}})$& \\
\hline
\addlinespace[0.3cm]
\hline
\multicolumn{6}{|c|}{$B_{3}$} \\
\hline
B1)&$\frac 12 (\a_1+\a_3)\in H(2)$ & $(\sll(2)\times \sll(2),\Delta\sll(2),0)$ & $\boldsymbol{-w\in I_2(1,3)}$  &$(\tt^1+\sll(2)+\tt^1,\tt^1+\sll(2)+\tt^1,\C^2_{\nicefrac12, -\nicefrac12})$    &   \\
\hline
B2)&$\boldsymbol{\w \in I_3(1)}$ &$(\tt^1+\sp(4),\tt^1+\sp(4),\C^4_{\nicefrac12})$ & $\boldsymbol{-\w \in I_1(3)}$ & $(\tt^1+\sp(4),\tt^1+\sp(4),\C^4_{-\nicefrac12})$ &\\
\hline
\multicolumn{6}{|c|}{$B_{4}$} \\
\hline
B3)&$\a_2+2\a_3+3\a_4 \in H(1)$ &$(\soo(7),G_2,0)$ & $\boldsymbol{-w\in I_1(4)}$ & $(\tt^1+\sll(4),\tt^1+\sll(4),\C^4_{-3})$      &   \\
\hline
B4)&$\frac 12 \a_1+\a_2+\frac 12 \a_3 \in H(4)$ & $(\sll(4),\sp(4),0)$ & $\boldsymbol{-\w \in I_4(2)}$ & $(\tt^1+\sp(4),\tt^1+\sp(4),\C^4_{-1})$ &  \\
 \hline
\multicolumn{6}{|c|}{$B_{n}$} \\
\hline
B5)&$\a_{k+1}+\dots+\a_n\in H(k)$ &$(\soo(2(n-k)+1),\soo(2(n-k)),0)$  & $\boldsymbol{-w\in}\linebreak \boldsymbol{I_k(k+1)}$  &$(\tt^1+\sll(k),\tt^1+\sll(k),\C^k_{-1}) $    &   \\
\hline
B6)&$\boldsymbol{\w\in I_{n-1}(n)}$ & $(\tt^1+\sll(n),\tt^1\sll(n),\C^n)$ & $\boldsymbol{-\w\in}\linebreak \boldsymbol{ I_n(n-1)}$ & $(\tt^1+\sll(2),\tt^1+\sll(2),\C^2_{-\nicefrac32})$ & \\
\hline
\addlinespace[0.3cm]
\hline
\multicolumn{6}{|c|}{$C_{n}$} \\
\hline
C1)&$\a_1+\dots+\a_{k}\in H(k+1)$ & $(\sll(k),\gll(k-1),0)$  & $\boldsymbol{-w\in}\linebreak\boldsymbol{ I_{k+1}(1,k)}$ & $(\tt^1+\sp(2n-2k),\tt^1+\sp(2n-2k),\C^{2n-2k}_{-1})$  &  $1\le k <n$ \\
\hline
C2) &$\a_2+2\a_3+\dots+2\a_{n-1}+\a_n \in H(1)$& $(\sp(2(n-1)),\sll(2)+\sp(2n-3),0)$ & $\boldsymbol{-\w\in I_1(3)}$ &$(\tt^1+\sll(3),\tt^1+\sll(3),\C^3_{-1})$ & \\
\hline
C3)&$\boldsymbol{\w \in I_k(k+1)}$ &$(\tt^1+\sll(k+1),\tt^1+\sll(k+1),\C^{k+1}_{-1})$ &  $\boldsymbol{-\w\in I_{k+1}(k)}$ &$(\tt^1+\sp(2(n-k)),\tt^1+\sp(2(n-k)),\C^{2n-2k})$ & \\
\hline
\addlinespace[0.3cm]
\hline
\multicolumn{6}{|c|}{$D_{4}$} \\
\hline
D1)&$\frac 12 (\a_i+2\a_2+\a_j) \in H(4)$ &$(\sll(4),\sp(k),0)$ & $\boldsymbol{-w\in I_k(2)}$ &$(\tt^1+\sll(2),\tt^1+\sll(2),\C^2_{-1})$     & $i,j,k\in 1,3,4; i \ne j \ne k$  \\
\hline
D2)&$\frac 12 (\a_i+\a_j) \in H(2)$ & $(\sll(2)\times \sll(2),\Delta\sll(2),0)$&  $ \boldsymbol{-\w \in I_2(i,j)}$ & $(\tt^1+\sll(3)+\tt^1,\tt^1+\sll(3)+\tt^1,\C^3_{-\nicefrac12,-\nicefrac12})$ & $i,j\in 1,3,4; i \ne j$\\
\hline
D3)&$\a_i+\a_2+\a_j \in H(k)$ &$(\sll(4),\gll(3),0)$& $\boldsymbol{-\w \in I_k(i,j)}$&$(\tt^1+\sll(3)+\tt^1,\tt^1+\sll(3)+\tt^1,\C^3_{1,-1})$ & $i,j,k\in 1,3,4; i \ne j \ne k$\\
\hline
\multicolumn{6}{|c|}{$D_{n}$} \\
\hline
D4)&$\a_{k+1}+\dots+\a_{n-2}+\frac 12 \a_{n-1}+\frac 12 \a_n \in H(k)$ &$(\soo(2(n-k)),\soo(2(n-k)-1),0)$\ & $\boldsymbol{-\w \in}\linebreak \boldsymbol{ I_k(k+1)}$& $(\tt^1+\sll(k),\tt^1+\sll(k),\C^{k}_{-1})$ &  \\
\hline
D5)&$\a_1+\dots + \a_{n-1} \in H(n)$ & $(\sll(n),\gll(n-1),0)$ & $\boldsymbol{-\w \in}\linebreak\boldsymbol{ I_n(1,n-1)}$ &$(\tt^1+\sll(n-1)+\tt^1,\tt^1+\sll(n-1)+\tt^1,\C^{n-1}_{1,-1})$ & or n and n-1 switched \\
\hline
D6)&$\boldsymbol{\w\in I_{n-1}(n)}$ & $(\sll(n),\sll(n),\C^n_{-\nicefrac12})$ & $\boldsymbol{-w\in}\linebreak\boldsymbol{ I_n(n-1)}$  &$(\sll(n),\sll(n),\C^n_{-\nicefrac12})$    & or n and n-1 switched  \\
\hline
\pagebreak
\hline
\multicolumn{6}{|c|}{$F_{4}$} \\
\hline
F1)&$\a_2+2\a_3+\a_4 \in H(1)$ &$(\sp(6),\sll(2)+\sp(4),0)$ & $\boldsymbol{-w\in I_1(3)}$   &$(\tt^1+\sll(3),\tt^1+\sll(3),\C^3_{-2})$   &   \\
\hline
F2)&$\a_3+\a_4\in H(2)$ & $(\sll(3),\gll(2),0)$ & $ \boldsymbol{-\w \in I_2(3,4)}$ &$(\tt^1+\sll(3)+\tt^1,\tt^1+\sll(3)+\tt^1,\C^3_{-1,-1})$& \\
\hline
F3)&$\a_1+\a_2+\a_3 \in H(4)$ &$(\soo(7),\soo(6),0)$ & $\boldsymbol{-\w \in I_4(1)}$ & $(\tt^1+\sp(6),\tt^1+\sp(6),\C^6_{-1})$ & \\
\hline
F4)&$\boldsymbol{\w \in I_3(2)}$ &$(\sll(3),\sll(3),\C^3_{2})$ & $\boldsymbol{-\w \in I_2(3)}$ &$(\tt^1+\sll(3),\tt^1+\sll(3),\C^3_{-2})$ & \\
\hline
\addlinespace[0.3cm]
\hline
\multicolumn{6}{|c|}{$G_{2}$} \\
\hline
G1)&$\a_1\in H(2)$ &$(\sll(2),0,0)$ & $\boldsymbol{-w\in I_2(1)}$  & $(\tt^1+\sll(2),\tt^1+\sll(2),\C^2_{-1})$    &   \\
\hline
\end{longtable}
}
\end{lis}

\begin{longtable}{c|c|c|c|c}
$
\begin{picture}(5000,3000)
\put(0,0){\usebox\dynkinathree}
\multiput(0,0)(3600,0){2}{\circle{600}}
\multiput(0,300)(3600,0){2}{\line(0,1){600}}
\put(0,900){\line(1,0){3600}}
\put(1600,-350){$\btl$}
 \end{picture}
$&$\begin{picture}(7200,3000)
\put(0,0){\usebox\dynkinafour}
\put(1800,0){\circle*{600}}
\put(5200,-250){$\btl$}
 \end{picture}$&$
\begin{picture}(9000,3000)
\put(0,0){\usebox\shortam}
\put(3600,0){\usebox\edge}
\put(5400,0){\usebox{\shortsusp}}
\put(7200,0){\usebox\dynkinatwo}
\put(5100,-350){$\btr$}
 \end{picture}
$&
$
\begin{picture}(9000,3000)
\put(0,0){\usebox\susp}
\put(3600,0){\usebox\dynkinatwo}
\put(5400,0){\usebox\susp}
\put(3300,-350){$\btl$}
\put(5100,-350){$\btr$}
 \end{picture}
$&$
\begin{picture}(9000,3000)(-1800,0)
\put(0,0){\usebox\dynkinbthree}
\multiput(0,0)(3600,0){2}{\circle{600}}
\multiput(0,300)(3600,0){2}{\line(0,1){600}}
\put(0,900){\line(1,0){3600}}
\put(1500,-350){$\btd$}
 \end{picture}
$\\
A1)&A2)&A3)&A4)&B1)\\
\hline
$
\begin{picture}(5000,3000)
\put(0,0){\usebox\dynkinbthree}
\put(3300,-350){$\btl$}
\put(-300,-350){$\btr$}
 \end{picture}
$&$
\begin{picture}(5000,3000)
\put(0,0){\usebox\dynkinbfour}
\put(5400,0){\circle*{600}}
\put(-300,-350){$\btr$}
\end{picture}
$
&
$
\begin{picture}(5000,3000)
\put(0,0){\usebox\dynkinbfour}
\put(1800,0){\circle*{600}}
\put(5100,-350){$\btl$}
\end{picture}
$&$
\begin{picture}(9000,3000)
\put(0,0){\usebox\shortsusp}
\put(1800,0){\usebox\dynkinatwo}
\put(3600,0){\usebox\shortsusp}
\put(5400,0){\usebox\dynkinbthree}
\put(3600,0){\circle*{600}}
\put(1500,-350){$\btl$}
\end{picture}
$
&
$
\begin{picture}(9000,3000)
\put(0,0){\usebox\dynkinatwo}
\put(1800,0){\usebox\shortsusp}
\put(3600,0){\usebox\dynkinbthree}
\put(5100,-350){$\btl$}
\put(6800,-350){$\btr$}
\end{picture}
$\\
B2) &B3) & B4) & B5) & B6)\\
\hline
&& $
\begin{picture}(9000,3000)
\put(0,0){\usebox\shortam}
\put(3600,0){\usebox\dynkinatwo}
\put(5400,0){\usebox\shortsusp}
\put(7200,0){\usebox\leftbiedge}
\put(5100,-350){$\btr$}
\end{picture}
$
&  
$
\begin{picture}(9000,3000)
\put(0,0){\usebox\dynkinathree}
\put(3600,0){\usebox\shortsusp}
\put(5400,0){\usebox\dynkincthree}
\put(-300,-350){$\btr$}
\put(3600,0){\circle*{600}}
\end{picture}
$
&
 $
\begin{picture}(9000,3000)
\put(0,0){\usebox\shortsusp}
\put(1800,0){\usebox\dynkinatwo}
\put(3600,0){\usebox\shortsusp}
\put(5400,0){\usebox\dynkincthree}
\put(3300,-350){$\btr$}
\put(1500,-350){$\btl$}
\end{picture}
$\\
&&C1)&C2)&C3)\\
\hline
 $
\begin{picture}(5000,3000)(0,-900)
\put(0,0){\usebox\dynkindfour}
\put(1200,800){\tiny $\nicefrac12$}
\put(1800,0){\circle*{600}}
\put(2600,1200){$\bt$}
\end{picture}
$
&
 $
\begin{picture}(5000,4000)(0,-900)
\put(0,0){\usebox\dynkindfour}
\put(0,0){\circle{600}}
\put(3000,1200){\circle{600}}
\put(3000,1500){\line(0,1){300}}
\put(3000,1800){\line(-1,0){3000}}
\put(0,300){\line(0,1){1500}}
\put(1400,-300){$\btd$}
\put(900,2200){\tiny $\nicefrac12$}
\end{picture}
$
& $
\begin{picture}(5000,3000)(0,-900)
\put(0,0){\usebox\athreene}
\put(2600,-1500){$\bt$}
\end{picture}
$
\\
D1) &D2) & D3)\\
\hline
&&
 $
\begin{picture}(9000,3000)
\put(0,0){\usebox\shortsusp}
\put(1800,0){\usebox\dynkinatwo}
\put(3600,0){\usebox\shortsusp}
\put(5400,0){\usebox\dynkindfour}
\put(3600,0){\circle*{600}}
\put(3300,600){\tiny$\nicefrac12$}
\put(1500,-350){$\btl$}
\end{picture}
$
&
 $
\begin{picture}(9000,3000)
\put(0,0){\usebox\amne}
\put(6300,-1500){$\bt$}
\end{picture}
$
&
 $
\begin{picture}(9000,3000)
\put(0,0){\usebox\susp}
\put(3600,0){\usebox\dynkindfour}
\put(6000,600){$\btd$}
\put(6000,-1500){$\bt$}
\end{picture}
$\\
&&D4)&D5)&D6)\\
\hline
$\begin{picture}(5000,3000)
\put(0,0){\usebox\dynkinffour}
\put(3600,0){\circle*{600}}
\put(-300,-350){$\btr$}
\end{picture}
$
&
$\begin{picture}(5000,3000)
\put(0,0){\usebox\dynkinffour}
\put(3600,0){\usebox\atwo}
\put(1500,-350){$\btl$}
\end{picture}
$&
$\begin{picture}(5000,3000)
\put(0,0){\usebox\dynkinffour}
\put(0,0){\circle*{600}}
\put(5100,-350){$\btl$}
\end{picture}
$
&
$\begin{picture}(5000,3000)
\put(0,0){\usebox\dynkinffour}
\put(1500,-350){$\btl$}
\put(3300,-350){$\btr$}
\end{picture}
$&$\begin{picture}(5000,3000)
\put(0,0){\usebox\dynkingtwo}
\put(0,0){\usebox\aone}
\put(1500,-350){$\bt$}
\end{picture}
$\\
F1&F2&F3) & F4) & G1) \\

\end{longtable}

\end{document}